\documentclass[reqno]{amsart}
\usepackage{epsf,epsfig}
\usepackage{graphics,color}
\usepackage{amsmath}
\usepackage{amssymb}
\usepackage{cite}
\usepackage{verbatim}
\usepackage{float}
\usepackage{graphicx}
\usepackage{amsthm}
\usepackage{textcomp}
\usepackage{subfig}
\usepackage{pgf}
\usepackage{ucs}
\usepackage[utf8x]{inputenc}

\newtheorem{theorem}{Theorem}[section]

\newtheorem{proposition}[theorem]{Proposition}

\theoremstyle{definition}

\numberwithin{equation}{section}
\renewcommand{\d}{\mathrm{d}}
\newcommand{\D}{\mathrm{D}}

\newcommand{\Rz}{{\mathbb R}}
\newcommand{\Nz}{\mathbb{N}}

\newcommand{\disp}{\displaystyle}
\newcommand{\haz}{\widehat}
\newcommand{\ove}{\overline}

\DeclareMathOperator*{\argmin}{arg\,min}

\newcommand\twoarrows{%
        \mathrel{\vcenter{\mathsurround0pt
                \ialign{##\crcr
                        \noalign{\nointerlineskip}$\Rightarrow$\crcr
                        \noalign{\nointerlineskip}$\not
                        \Leftarrow$\crcr 
                }%
        }}%
}
\newcommand{\e}{{\rm e}}

\begin{document}  

\title[Two time discretizations schemes for gradient flows]{Two 
  structure-preserving  
time discretizations for gradient flows}
\author[A. J\"ungel]{Ansgar J\"ungel}
\address[Ansgar J\"ungel]{Institute for Analysis and Scientific Computing, Vienna University of Technology, Wiedner
Hauptstra\ss e 8-10, 1040 Wien, Austria}
\email{juengel@tuwien.ac.at}
\urladdr{http://www.asc.tuwien.ac.at/$\sim$juengel}
\author[U. Stefanelli]{Ulisse Stefanelli} 
\address[Ulisse Stefanelli]{Faculty of Mathematics, University of Vienna, 
Oskar-Morgenstern-Platz 1, 1090 Wien, Austria  and  Istituto di Matematica
Applicata e Tecnologie Informatiche \textit{{E. Magenes}}, v. Ferrata 1, 27100
Pavia, Italy.}
\email{ulisse.stefanelli@univie.ac.at}
\urladdr{http://www.mat.univie.ac.at/$\sim$stefanelli}
\author[L. Trussardi]{Lara Trussardi}
\address[Lara Trussardi]{Faculty of Mathematics, University of Vienna, 
Oskar-Morgenstern-Platz 1, 1090 Wien, Austria.}
\email{lara.trussardi@univie.ac.at}
\urladdr{http://www.mat.univie.ac.at/$\sim$trussardi}

\date{\today}


\begin{abstract}
 
 The equality between dissipation and energy drop is a structural property
of gradient-flow dynamics. 
The classical implicit Euler scheme
fails to reproduce  this equality at the discrete level.  We discuss
two modifications of the Euler
scheme satisfying an exact energy equality at the discrete
level. Existence of discrete solutions and their convergence as the
fineness of the partition goes to zero are discussed. Eventually, we
address extensions to generalized gradient flows, GENERIC
flows, and curves of maximal slope in metric spaces.
\end{abstract}

\maketitle


\section{Introduction} 

Gradient flows are the paradigm of dissipative
evolution and arise ubiquitously in
applications \cite{Ambrosio08}. In abstract terms, they are formulated
as  
\begin{equation}
  \label{eq:gf}
  u'(t) + \D \phi(u(t)) = 0 \quad \text{for a.e.} \ t \in (0,T), \quad u(0)=u_0,
\end{equation}
where the solution $t \mapsto u(t)\in H$ is a trajectory in a Hilbert
space $H$, the potential $\phi$ is given, and $u_0$ is
a prescribed initial datum. We assume $\phi$ to be smooth for the
purpose of this introductory discussion, so that $\D \phi$ is here
the Fr\'echet differential. Let us however anticipate
that the analysis will encompass
nonsmooth situations as well.

Problem \eqref{eq:gf} arises in a variety of applications including heat
conduction, the Stefan problem, the Hele-Shaw cell, porous media, parabolic
variational inequalities, some classes of ODEs with obstacles, degenerate
parabolic PDEs, and the mean curvature flow for Cartesian graphs, among many
others \cite{Noch-Sav-Verdi2}. Correspondingly,  \eqref{eq:gf} has
attracted constant attention during the last half century, starting
from the seminal contributions by {\sc K\= omura} \cite{Komura67}, {\sc
  Crandall-Pazy}\cite{Crandall-Pazy69}, and {\sc Br\'ezis}
\cite{Brezis71,Brezis73}.

Solutions of \eqref{eq:gf} fulfill the {\it
  energy equality} 
  \begin{equation}
    \label{eq:balance}
    \phi(u(t)) + \int_s^t \| u'(r)\|^2\d r =\phi(u(s))   
  \end{equation}
for all $[s,t]\subset [0,T]$. This equality expresses the fact that the energy drop $ \phi(u(s))
-\phi(u(t))$  along the trajectory in the time interval $[s,t]$
{\it exactly} corresponds to the squared $L^2$ norm of the velocity
$u'$. Here and in the following, $\|\cdot\|$ denotes the norm in $H$.
For all  $u$ solving \eqref{eq:gf}, the energy equality \eqref{eq:balance} can
be rewritten as
\begin{equation}
    \label{eq:generic}
    \phi(u(t)) + \frac12\int_s^t \| u'(r)\|^2\d r +\frac12\int_s^t
    \| \D \phi(u(r))\|^2\d r =\phi(u(s))  
  \end{equation}
for all $[s,t]\subset [0,T]$. In particular, \eqref{eq:balance} and \eqref{eq:generic} are equivalent for
solutions to \eqref{eq:gf}. 

If $u$ does not solve \eqref{eq:gf},
conditions \eqref{eq:balance} and \eqref{eq:generic} are not
equivalent anymore. On the one hand, relation \eqref{eq:generic}
and $u(0)=u_0$ imply that
\begin{align*}
 & \frac12\int_0^T\|u' + \D \phi(u)\|^2 \d r
=  \frac12\int_0^T\left(\|u' \|^2 +2(\D\phi(u),u')
  + \|\D \phi(u)\|^2\right)\d r \nonumber\\
&= \frac12\int_0^T\left(\|u' \|^2 + 2(\phi\circ u)'
  + \|\D \phi(u)\|^2 \right)\d r \nonumber\\
&= \phi(u(T))-\phi(u(0)) +
  \frac12\int_0^T\|u' \|^2\d r +\frac12 \int_0^T\|\D \phi(u)\|^2\d r = 0,
\end{align*}
so that $u$ solves \eqref{eq:gf} and hence fulfills
\eqref{eq:balance} as well. 

On the other hand, \eqref{eq:balance} does not imply
\eqref{eq:gf}. In order to check this, take $\phi(u)=u_1+u_2$ for all
$u=(u_1,u_2)\in \Rz^2=H$ and letting $u(t)=(-t,0)$ for all $t\geq
0$, we see that  $u$ fulfills \eqref{eq:balance} but does not solve \eqref{eq:gf}.
All in all, we have verified that 
$$\eqref{eq:gf} \ \Leftrightarrow \ \eqref{eq:generic} \  \twoarrows \ 
\eqref{eq:balance} .$$

The analysis of Problem \eqref{eq:gf} often relies on a time discretization.
Let a partition $\{0=t_0<t_1<\dots<t_N=T\}$ be given and indicate by $\tau_i =
t_i - t_{i-1}$ its time step. A classical
discretization  of \eqref{eq:gf} 
is the {\it implicit Euler scheme}. Given $u_0 $, this reads as
$$ \frac{u_i - u_{i-1}}{\tau_i} + \D \phi (u_i)= 0\quad
\text{for} \ \ i=1,\dots,N, $$
which can be equivalently reformulated in variational terms as
$$ u_i \in \argmin_{u \in H} \left
  (\phi(u)+\frac{\tau_i}{2}\left\|\frac{u-u_{i-1}}{\tau_i}\right\|^2\right)\quad
\text{for} \ \ i=1,\dots,N.$$
Minimality implies stability of the scheme. Under
appropriate assumptions on $\phi$, convergence follows
as the fineness of the partition goes to zero.
 
Unfortunately, the Euler scheme fails to reproduce exact
dissipation dynamics at the discrete level: 
The dissipation $\|(u_i-u_{i-1})/\tau_i\|^2$ of the $i$-th time step does
not correspond to the energy drop $\phi(u_{i-1})- \phi(u_{i})$ in
the same step. Indeed, 
energy-{\em dissipating} schemes for gradient systems were developed quite intensively 
in the literature. Examples include the convex splitting method,
popularized in \cite{Eyr98}, algebraically stable Runge-Kutta methods
\cite{HaLu14}, and the discrete variational derivative method
\cite{FuMa11}. The exact replication of the dissipation dynamics
is nevertheless a desirable feature in view of developing 
structure-preserving algorithms.

The aim of this note is to analyze two variants of the Euler scheme
reproducing at the discrete level the energy equalities
\eqref{eq:balance} and \eqref{eq:generic}, respectively.
The first of such modifications can be traced back at least to {\sc Gonzalez} \cite{
Gonzales}, see also \cite{condette,Romero}, and will hence be termed {\it Gonzalez
scheme} here. Given $u_0$, for all $i=1,\dots,N$, the Gonzalez
scheme defines $u_i$ starting from the previous value $u_{i-1}$ by
letting $u_i  = u_{i-1} $ in case $\D \phi(u_{i-1})=0$ and solving 
   \begin{align}
  &  \frac{u_i - u_{i-1}}{\tau_i} + \D \phi(u_i)  \nonumber\\
  &+
    \big(\phi(u_i)- \phi(u_{i-1}) - (\D \phi(u_i), u_i
    - u_{i-1})\big)\frac{u_i - u_{i-1}}{\|u_i -
      u_{i-1}\|^2} = 0 \label{eq:romero02}
  \end{align} 
if $\D \phi(u_{i-1})\not =0$,
  where $(\cdot,\cdot)
$ is the scalar product in
$H$. Indeed, by taking the scalar product of \eqref{eq:romero02} with
$u_i - u_{i-1} $, the terms containing $\D \phi(u_i)$
cancel out and we are left with 
\begin{equation*}
  \phi(u_i) + \tau_i \left\| \frac{u_i - u_{i-1}}{\tau_i}\right\|^2= \phi(u_{i-1}).
\end{equation*}
By summing up the latter over $i$, one obtains a  discrete version of
the energy balance \eqref{eq:balance}. This implies in particular the
stability of the Gonzalez scheme. 
We remark that the original Gonzalez scheme uses $\D\phi((u_i+u_{i-1})/2)$
instead of $\D\phi(u_i)$ in \eqref{eq:romero02}, leading in both cases to an 
implicit scheme. To our knowledge, no 
analysis for the solvability of \eqref{eq:romero02} nor for the
convergence of the scheme  has been presented so far. We fill this gap in
Section \ref{sec:generic}, where we prove that for sufficiently smooth functions
$\phi$ and finite-dimensional spaces $H$, the Gonzalez scheme
admits solutions (Theorem \ref{thm:r1}) that converge with an 
explicit and sharp a-priori rate (Theorem \ref{thm:r2}). 

The second modification of the Euler scheme reproduces at the
discrete level relation \eqref{eq:generic}. 
The scheme is inspired by  the approach to steepest-descent dynamics
by {\sc De Giorgi}~\cite{DeGiorgi}, and we hence refer to it as {\it De
  Giorgi scheme} in the following. Given $u_0$, the De Giorgi scheme reads as
\begin{equation}
  \label{eq:dg}
  \phi(u_i) + \frac{\tau_i}{2}\left\| \frac{u_i - u_{i-1}}{\tau_i}\right\|^2
    +\frac{\tau_i}{2}\| \D\phi(u_i)\|^2 - \phi(u_{i-1})=0
\end{equation}
for $i=1,\dots,N$.
By summing up the latter relation over $i$, one readily obtains a discrete
version of \eqref{eq:generic}, which directly proves stability. The De Giorgi scheme reduces to a single scalar equation and allows for an existence (Theorem
\ref{thm:existence}) and a convergence proof (Theorem \ref{thm:convergence}),
even in nonsmooth and infinite-dimensional situations, see Section
\ref{sec:dg}. In addition, sharp error estimates can be derived in the
finite-dimensional case. 

Instead of solving \eqref{eq:dg}, one could simply minimize its
residual by minimizing the functional defined by  the left-hand side. 
As a by-product of our analysis, we show that this
alternative variational approach with respect to the Euler scheme can
also be considered, giving convergence and, when restricted to finite
dimensions, optimal error bounds. 

Remarkably, the De
Giorgi scheme can serve as an a-posteriori tool to check the convergence
of time-discrete approximations $u_i$, regardless of the method used to
generate them. Indeed, the analysis of
the (positive parts of the) residuals to \eqref{eq:dg} can reveal
whether the approximation converges to the unique solution of
\eqref{eq:gf}, see condition \eqref{eq:condi} below.

Eventually, the De Giorgi scheme can be extended to other nonlinear 
evolution settings. In Section
\ref{sec:ext} we comment on its application to the case of 
 generalized gradient flows, GENERIC flows, and curves of maximal slope in
metric spaces.

Before moving on, let us briefly put our contribution in context. 
Numerical schemes conserving first integrals (in particular, the energy)
can be found in \cite{Gonzales,QuMc08,Romero2,Romero3}.
See \cite{SMSF15} for a recent reference and  
\cite{CGMMOOQ12} for a review. In particular, the Gonzalez scheme falls within 
the general class of  {\it discrete-gradient methods} along with the choice 
\begin{equation*}
  \nabla_d\phi(u,v) :=
      \big(\phi(u){-} \phi(v) - (\D \phi(u), u-v)\big)\displaystyle\frac{u-v}{\|u-v\|^2}
\end{equation*}
for the {\it discrete gradient}, for $u \not = v $.
These methods are specifically tailored to exactly reproduce
dissipation dynamics. The discrete gradient is designed to satisfy
a discrete chain rule, which in turn delivers numerical integrators replicating the
dissipation property of the continuous system 
\cite{SMSF15}. In the specific case of the Gonzalez scheme, such
discrete chain rule reads
(see \cite{Gonzales} or \cite[Formula (5.10)]{HLW06})
$\phi(u)-\phi(v) = \nabla_d\phi(u,v)\cdot(u-v).$
A systematic study of discrete gradient methods can be found in 
\cite{Gonzales,MQR99}. 

Discrete-gradient methods have been applied to a large class of equations, not
necessarily of gradient-flow type. Examples are schemes for conservative 
partial differential equations fulfilling a discrete conservation law \cite{McQu14}
and linearly-fitted numerical schemes for conservative
wave equations \cite{LWS17}. The latter method also preserves the dissipation
structure of dissipative wave equations. Moreover, the discrete gradient method 
was applied to subdifferentials in \cite{BKS08}.

As the implicit Euler method is of first order only, one may look for
higher-order variational schemes. Two second-order schemes, the variational
implicit midpoint and the extrapolated variational implicit Euler schemes, 
are proposed in \cite{LeTu17}. A variational BDF2 
(two-step Backward Differentiation Formula) method is analyzed in \cite{MaPl17}.
While second-order convergence is expected in a smooth Hilbert space case,
the convergence of rate one-half is shown in the general metric setting. 
All the mentioned results do not replicate the exact energy
dissipation dynamics of the gradient flow problem.

One may ask whether the Gonzalez and De Giorgi scheme, analyzed in this paper, 
can be extended to give higher-order convergence. 
While higher-order consistency and convergence
have been proved for discrete gradient methods for ODEs and the discrete
variational derivative method for PDEs \cite[Section 1.4]{FuMa11},
we are not aware of higher-order generalizations of our schemes.
Moreover, as the result of \cite{MaPl17} shows, the lack of smoothness
may decrease the optimal convergence order. For quadratic potentials $\phi$,
the G-stability by Dahlquist allows for energy-dissipating schemes, but
requires to redefine the energy as a function of $(u_i,u_{i-1})$ instead of
$u_i$ alone \cite{JuMi12}. General energy-dissipative Runge-Kutta schemes 
are studied in \cite{JuSc17}, but again, they do not replicate the  exact energy
dissipation dynamics of the problem.

 Our primary focus is that of approximating the
limiting continuous dynamics. Still, one has to mention that the discrete schemes discussed here may
bear some interest in connection with optimization as well.  Indeed,
the unconstrained minimization of the potential $\phi$ can be tackled
by considering iterative approximations. In the case of a convex
$\phi$, the easiest example in this class is the
{\it proximal algorithm} $u_i = \text{prox}_{\tau_i\phi}(u_{i-1})$, where
$\text{prox}_{\tau_i\phi}=({\rm id} + \tau_i\partial \phi)^{-1}: H \to H$. This corresponds to iterations of the
implicit Euler method, starting from an initial guess $u_0$. If the
set of minimum points ${\rm
  argmin}\, \phi $ is not empty and $\sum\tau_i=\infty$, the proximal
algorithm weakly converges to a minimizer \cite[Props.~6.1-2, pp. 99-100]{Pey}

A variety of modifications of the proximal algorithm has been
advanced in order to solve optimization problems with specific
structure, or with higher efficiency \cite{BC,Bert}. A possibility is
that of augmenting the input of the proximity operator
$\text{prox}_{\tau_i\phi}$ by specific corrections. This is indeed the
case of the Gonzales scheme, for it corresponds to the
position
$$u_i = \text{prox}_{\tau_i\phi}\left(u_{i-1} -\tau_i
  \nabla_d\phi(u_i,u_{i-1}) \right).$$ 
 Note, however, that the small
correction $\tau_i\nabla_d\phi(u_i,u_{i-1})  $ depends on $u_i$ as well, which ultimately requires an implicit
updating step. A discussion on the relation between discrete and continuous gradient-flow
dynamics in presence of higher-order corrections can be
found in \cite{Shi}. 

Starting from Nesterov's result \cite{Nesterov},
acceleration methods based on inertial gradient systems of the form 
$$u'' (t) + \alpha(t)u'(t)+ \D \phi(u(t))\ni 0$$
where $\alpha$ is a suitably chosen damping coefficient vanishing for $t\to\infty$, are currently
attracting attention,
see  
\cite{Attouch17,Attouch18,Bot,Cabot09} among many others. One could
consider extensions of the Gonzales and the De Giorgi schemes to the
case of
inertial gradient systems as well. In order to replicate
the equality between dissipation and energy drop, now
including the kinetic energy, one has to introduce a higher-order correction in the
time-increment term as well.
We shall develop these considerations elsewhere.

In order to make the connection with optimization algorithms complete,
one would have to study the convergence of the discrete sequence $u_i$
as $i\to \infty$. This would correspond to investigate the long-time
discrete
dynamics and eventually connecting it with the corresponding continuous
counterpart. Although some of our results hold on the whole time
semiline as well, such a long time analysis is presently beyond the scope
of this paper. We present some evidence of the potential of
these methods in a comparative discussion in Section
\ref{sec:compare} below.


\section{Preliminaries}

Let us introduce our notation and recall some basic results.
Denote by $H$ a real separable Hilbert space with scalar
product $(\cdot,\cdot)$ and associated norm $\| \cdot \|$. Let
$\phi: H \to (-\infty,\infty]$ be a proper and lower semicontinuous functional,
not necessarily convex, and let $u_0\in D(\phi)=\{u \in H \ : \ \phi(u) \not = \infty\}$
({\it essential domain}). The  {\it Fr\'e\-chet subdifferential} $\partial \phi(u)$
of $\phi$ at $u \in
D(\phi)$ is the set of points $\xi \in H$ such that 
$$\liminf_{v \to u}\frac{\phi(v) - \phi(u) - (\xi,v-u)}{\| v-u
  \|}\geq 0.$$
The latter reduces to the classical gradient ${\rm D}\phi(u)$ in case
$\phi$ is Fr\'e\-chet differentiable at $u$ and to the subdifferential
in the sense of convex analysis if $\phi$ is convex. If $\phi=\phi_1 +
\phi_2$, where $\phi_1$ is convex, proper, and lower semicontinuous, and
$\phi_2\in C^1(H)$, we have $\partial \phi=\partial \phi_1 + \D \phi_2$. 
We denote by $D(\partial \phi)$
the essential domain $D(\partial \phi) = \{u \in H \ : \ \partial
\phi(u) \not = \emptyset\}$ and remark that $\partial\phi(u)$ is
either convex or empty.

The well-posedness of Problem \eqref{eq:gf} for convex functions
$\phi$ is classical \cite{Brezis73}.  
Nonconvexity can also be
allowed, as long as $\phi$ has compact sublevels, see
{\sc Rossi \& Savar\'e} \cite{Rossi-Savare06} for some comprehensive
existence theory. Crucial assumptions are the conditional
strong-weak closeness of  $\partial \phi$ and the
validity of the chain rule.
These hold, for instance, if $\phi$ is a $C^{1,1}$
perturbation of a convex functional,   which includes the case of $\lambda$-convex
functionals \cite{Ambrosio08}. In addition, a suitable class of
dominated, not $C^{1,1}$ perturbations can be considered as well
\cite{Rossi-Savare06}.  For the sake of definiteness, we focus
here on the finite-time case $t \in [0,T]$. Let us, however, mention
that some of our results hold
for gradient flows on the semiline $t \in [0,\infty)$ as well. We will
comment on this issue along the text. 

Existence results for gradient flows often take the moves from time
discretizations.
 In the following, we will consider families of
partitions
$$\{0=t_0^n<t_1^n<\dots<t^n_{N^n}\}$$
of the interval $[0,T]$, indicate their time steps by $\tau^n_i =
t^n_i - t^n_{i-1}$, and denote their fineness by $\tau^n = \max_i
\tau^n_i$. Correspondingly, one is interested in finding a
time-discrete solution $u^n_i\in D(\partial \phi)$ with $u_0^n=u_0$ via some 
time-discrete scheme. 

By defining the affine functions $\ell_j^n (t) = (t-t_{j-1}^n)/\tau_j^n$,
for any given vector $w_i^n$, $i=0,\dots, N^n$, we use the notation 
$\haz w^n, \, \ove w^n : [0,\infty) \to H$  for the piecewise affine and
the backward constant interpolants of the values $w_i^n$ on the partition, namely
\begin{align*}
  &\haz w^n(0)=\ove w^n(0)=w_0^n, \quad \haz w^n(t) = \ell_j^n(t)w_j^n +
  (1-\ell_j^n(t))w_{j-1}^n, \\
& \ove w^n(t)
= w_j^n \quad \mbox{for all }t \in (t_{j-1}^n,t_j^n].
\end{align*}

By letting $\tau^n\to 0$ one is then asked to show the convergence of
the time-discrete trajectory $\haz u^n$ to a solution to \eqref{eq:gf}.

A {\it caveat} on notation: In the following, we use the same symbol $C$ to indicate a
generic positive constant, possibly depending on the data and varying from line to
line.  Occasionally, we may explicitly indicate
dependences of a constant by subscripts. 

\section{The Gonzalez scheme}\label{sec:generic}
 
As mentioned in the introduction, the Gonzalez scheme can be traced
back at least to
\cite{Gonzales} and has been already applied in various thermomechanical
contexts \cite{Betsch,Garcia,Gonzales2,Romero,Romero2,Romero3}. 
We aim here at providing some theoretical analysis by focusing on
solvability (Theorem \ref{thm:r1}) and convergence (Theorem \ref{thm:r2}). 
Before moving on, let us equivalently
rewrite relation \eqref{eq:romero02} as
\begin{equation}
  \label{eq:romero2}
  \phi(u_i) + \tau_i \left\| \frac{u_i - u_{i-1}}{\tau_i}\right\|^2 =
  \phi(u_{i-1}) \ \ \text{and}
  \ \ \D \phi(u_i) \ \ \text{is parallel to} \ \ u_i - u_{i-1}.
\end{equation}
Indeed, a solution to \eqref{eq:romero02} fulfills \eqref{eq:romero2} as
well: The energy equality holds and one reads off the parallelism by
comparing the terms in \eqref{eq:romero02}. On the contrary, if $u_i$
solves \eqref{eq:romero2}, we compute
\begin{align*}
  0 &= \left(\phi(u_i) -\phi(u_{i-1}) + \tau_i \left\| \frac{u_i -
  u_{i-1}}{\tau_i}\right\|^2\right)
  \frac{u_i-u_{i-1}}{\|u_i-u_{i-1}\|^2}\\
&=
\left(\phi(u_i) -\phi(u_{i-1}) + \tau_i \left\| \frac{u_i -
  u_{i-1}}{\tau_i}\right\|^2\right) \frac{u_i-u_{i-1}}{\|u_i-u_{i-1}\|^2}
  \\
&\phantom{xx}{}+ \frac{(\D\phi(u_i),u_i - u_{i-1})}{\|u_i-u_{i-1}\|}
  \frac{u_i-u_{i-1}}{\|u_i-u_{i-1}\|} - \frac{(\D\phi(u_i),u_i - u_{i-1})}{\|u_i-u_{i-1}\|}
  \frac{u_i-u_{i-1}}{\|u_i-u_{i-1}\|} \\
&=\left(\phi(u_i) -\phi(u_i) + \tau_i \left\| \frac{u_i -
  u_{i-1}}{\tau_i}\right\|^2\right) \frac{u_i-u_{i-1}}{\|u_i-u_{i-1}\|^2}
  \\
&\phantom{xx}{}+  \D\phi(u_i) - \frac{(\D\phi(u_i),u_i - u_{i-1})}{\|u_i-u_{i-1}\|}
  \frac{u_i-u_{i-1}}{\|u_i-u_{i-1}\|} 
\end{align*}
which is exactly \eqref{eq:romero02}. We will make use of this
equivalent formulation to prove the following result.

\begin{theorem}[Existence for the Gonzalez scheme]\label{thm:r1} Let $\phi \in
  C^1(\Rz^d)$ be bounded from below and let $\D \phi(u_{i-1}) \not
  =0$. Then there exists $u_i \in \Rz^d\setminus \{u_{i-1}\}$ such that
  \eqref{eq:romero02} holds.
\end{theorem}

  \begin{proof}
 Define $K=\{u \in \Rz^d \ : \ \phi(u)   + \tau_i \|(u-u_{i-1})/\tau_i\|^2
  = \phi(u_{i-1})\} $. Note that $K$ is not empty as $u_{i-1}\in
  K$. Denote by $f$ the continuous function $f(u) = \| u - u_{i-1}\|^2/\tau_i$. 
	The set $K$ is compact, for it is the preimage of
 $\{\phi(u_{i-1})\}$ with respect to the  continuous and coercive
 function $g= \phi + f$.
One can hence define $u_i$ as the maximizer of   $f $
 on $K$. The element $u_i$ is
  surely different from $u_{i-1}$ as $K$ does not reduce to $u_{i-1}$
	because of $\D \phi(u_{i-1})\not = 0$ and the implicit function theorem.

We now check that $u_i$ solves \eqref{eq:romero2}. The
energy equality follows directly from the fact that $u_i \in K$ and
one is left to verify that $\D \phi(u_i)$ and $u_i -
u_{i-1}$ are parallel. As $\D \phi(u_i) = \D g(u_i)-\D f(u_i) $, such
parallelism follows once we prove that $\D g(u_i)$ and $\D f(u_i)
$ are parallel, as we have $\D f(u_i)= 2(u_i -
u_{i-1})/\tau_i$.

In case $\D g(u_i)=0$, the
parallelism trivially holds. If $\D g(u_i)\not =0$, again the implicit function theorem
ensures that $K$ is a $C^1$ hypersurface in a neighborhood of $u_i$ and
that $\D g(u_i)$ is orthogonal to $K$ at $u_i$. Assume by contradiction that $\D
f(u_i)$ is not
parallel to $\D g(u_i)$.  Then the projection $\nu$ of $\D f(u_i)$ onto
the tangent space $\D g(u_i)^\perp$ is nonzero. By letting
$\gamma:[-\delta,\delta] \to K$ be a  
$C^1$ curve for some small $\delta>0$ with $\gamma(0)=u_i$ and
 $\gamma'(0)=\nu$, we deduce that $f\circ \gamma$ is not maximized
 at $0$, contradicting the maximality of $u_i$. 
  \end{proof}

By inspecting the proof of Theorem \ref{thm:r1}, one realizes that the lower bound on
$\phi$ can be replaced by the weaker quadratic bound $\phi(u) \geq
-C\|u\|^2-C$ at the price of prescribing the time step $\tau_i$ to be
small enough. Indeed, what is needed in the proof is just  the coercivity of the
function $ u\mapsto \phi(u) + \tau_i \|(u-u_{i-1})/\tau_i\|^2$,
which holds for sufficiently small $\tau_i>0$ in view of
$$
  \phi(u) + \tau_i\bigg\|\frac{u-u_{i-1}}{\tau_i}\bigg\|^2
	\ge \bigg(-C+\frac{1}{2\tau_i}\bigg)\|u\|^2 
	- \bigg(C + \frac{1}{\tau_i}\|u_{i-1}\|^2\bigg).
$$

Let a sequence of partitions $\{0=t_0^n<t^n_1<\dots<t^n_{N^n}=T \}$ be given
with $\tau^n=\max_i\tau_i^n =\max_i (t_i^n -t_{i-1}^n)\to 0$ 
and correspondingly, let $u_i^n$ be
solutions of the Gonzalez scheme, whose existence is guaranteed by 
Theorem \ref{thm:r1}. We have the following convergence result.
  
\begin{theorem}[Convergence for the Gonzalez scheme]\label{thm:r2} 
  Let $\phi\in C^1(\Rz^d)$ be
  bounded from below.
	\begin{enumerate}
	\item[\rm (i)] There exists a subsequence which is not relabeled
	such that $\haz u^n \to u$  weakly in $H^1(0,T;\Rz^d)$ as $n\to\infty$, 
	where $u$ solves
  \eqref{eq:gf}. 
	\item[\rm (ii)] Let $\phi\in C^{1,1}_{\rm loc}(\Rz^d)$. Then the whole sequence
  $(\haz u^n)$ converges strongly in $W^{1,\infty}(0,T;\Rz^d)$ and
	we have the error bound
  \begin{equation}
    \label{eq:error_romero}
    \|u-\haz u^n\|_{W^{1,\infty}(0,T;\Rz^d)}\leq C \tau^n. 
  \end{equation}
\item[\rm (iii)] Let $\phi\in C^3(\Rz^d)$ and assume that the condition 
\begin{equation}
  \label{eq:ansgar}
  \D^2\phi(v)w \ \ \text{is parallel to} \ \ w \ \ \text{for any} \ v,\, w \in \Rz^d
\end{equation}
holds. Then 
 \begin{equation}
    \label{eq:error_romero2}
    \|u - \haz u^n\|_{W^{1,\infty} (0,T;\Rz^d)}\leq C (\tau^n)^2.
  \end{equation}
	\end{enumerate}
\end{theorem}

\begin{proof} 
{\em Step 1: Proof of {\rm (i)}}.
Let us start by verifying the stability of the Gonzalez scheme. Indeed,
by adding up the local energy equality, we obtain
\begin{equation} \phi(u^n_m) + \sum_{i=1}^{m} \tau_i \left\|\frac{u_i^n -
    u^n_{i-1}}{\tau_i} \right\|^2 =\phi(u_0)\label{eq:rombound}
\end{equation}
for all $m\leq N^n$. This implies that $(\haz u^n)$ is bounded in
$H^1(0,T;\Rz^d)\hookrightarrow L^\infty(0,T;\Rz^d)$ independently of $n$. 
In particular, $\|\haz u^n(t)\|$  and $\| \ove u^n(t)\|$ are 
bounded for all times.

Let us rewrite the
Gonzalez scheme in the compact form 
\begin{equation}
  \label{eq:romerocompact}
  (\haz u^n)' + \D \phi(\ove u^n) + \ove r^n =0, 
\end{equation}
where the remainder $\ove r^n$ is defined by
\begin{align*}
  &r^n_i:=0 \ \ \text{if}  \ \D\phi(u_{i-1}^n)=0 \ \ \text{and}\\
&r^n_i:= \Big(\phi(u_i^n){-} \phi(u_{i-1}^n) {-} \big(\D \phi(u_i^n), u_i^n
    {-} u_{i-1}^n\big)\Big)\disp\frac{u_i^n {-} u_{i-1}^n}{\|u_i^n {-}
      u_{i-1}^n\|^2}  \ \ \text{if} \ \D\phi(u_{i-1}^n)\not =0.
\end{align*}
One readily checks that 
\begin{align*}
  \|r^n_i\| \leq \left| \int_0^1 \left(\D \phi(\xi u_i^n {+} (1{-}\xi) u_{i-1}^n)
    - \D \phi(u_i^n), \frac{u_i^n {-} u_{i-1}^n}{\|u_i^n {-}
      u_{i-1}^n\|}\right) \d \xi \right|.
\end{align*}
Denote by $\omega$ a continuity modulus for $\D \phi$
on a ball containing $\ove u^n$ for all times. In particular,
$\omega:[0,\infty)\to [0,\infty)$ is nondecreasing and $\lim_{t\to 0+}
\omega(t) = \omega (0)=0$. Then
$$
  \|r^n_i\|\leq \omega\big((1-\xi) \|u_i - u_{i-1}\|\big)
	\leq \omega \big( \|u_i - u_{i-1}\|\big).
$$
This implies that
\begin{align*}
  \max_{i=1,\dots, N^n} \|r^n_i\| 
	&\leq \max_i \,\omega \big( \|u_i - u_{i-1}\|\big) 
	\leq \max_i\omega \left(\int_{t_{i-1}^n}^{t^n_i}\|(\haz u^n)'\|\d r\right) \nonumber\\
  &\leq \max_i\,\omega
  \left((t_{i}^n-t^n_{i-1})^{1/2}\left(\int_{t_{i-1}^n}^{t^n_i}\|(\haz
    u^n)'\|^2\d r\right)^{1/2}\right) \nonumber\\
  &\leq \max_i\,\omega
  \left((t_{i}^n-t^n_{i-1})^{1/2}\left(\int_0^T\|(\haz
    u^n)'\|^2\d r\right)^{1/2}\right) \nonumber\\
  &\leq \omega(C(\tau^n)^{1/2}).
\end{align*}
In particular, $\ove r^n \to 0$ strongly in $L^\infty(0,T;\Rz^d)$. 

Bound \eqref{eq:rombound} ensures, upon passing to subsequences, 
which are not relabeled, that
\begin{align*}
  \haz u^n \rightharpoonup u \quad& \text{weakly in} \ \ H^1(0,T;\Rz^d),\\
\haz u^n, \, \ove u^n \to u \quad& \text{strongly in} \ \
L^\infty(0,T;\Rz^d),\\
\D\phi(\ove u^n) \to \D\phi(u) \quad&\text{strongly in} \ \
L^q(0,T;\Rz^d) \quad \mbox{for all }q <\infty.
\end{align*}
For the last convergence, we have used that $\D\phi(\ove u^n) \to
\D\phi(u)$ pointwise, as $\D \phi$ is continuous, and that
$\|\D\phi(\ove u^n) \|$ is uniformly bounded.
One can hence pass to the limit in \eqref{eq:romerocompact} and find
that $u$ solves the gradient-flow problem \eqref{eq:gf}.

{\em Step 2: $\phi\in C^{1,1}_{\rm loc}(\Rz^d)$}.
In case $\D \phi$ is locally Lipschitz continuous, the solution of \eqref{eq:gf} is unique and we can prove an error estimate. Let
us start by noting that in this case one has
\begin{equation}  
    \|r^n_i\| \leq C \|u_i^n - u^n_{i-1}\| =  C \tau_i \left\| \frac{u_i^n -
    u^n_{i-1}}{\tau_i}\right\|.\label{eq:rr}
\end{equation}
As $\ove u^n$ is bounded, \eqref{eq:romerocompact} and \eqref{eq:rr} imply that  
$$ 
  \|(\haz u^n)'\| \leq \|\D \phi(\ove u^n)\| + \|\ove r^n\| \leq C + C
  \tau^n  \|(\haz u^n)'\|. 
$$
Consequently, for sufficiently small values of $\tau^n>0$, 
we infer that $(\haz u^n)'$ is
bounded in $L^\infty(0,T;\Rz^d)$. Relation \eqref{eq:rr} then ensures that  
\begin{equation}
\| \ove r^n \|_{L^\infty(0,T;\Rz^d)} \leq C\tau^n\| (\haz u^n)'
\|_{L^\infty(0,T;\Rz^d)} \leq C\tau^n. \label{eq:r}
\end{equation}

Take now the difference between
\eqref{eq:romerocompact}, written for the partition $n$, and the same
equation for the partition $m$, test it against $\haz u^n - \haz u^m$, and
integrate in time. Then, in view of the Lipschitz continuity of $\D\phi$,
\begin{align}\label{eq:un-um}
  \frac12& \|(\haz u^n - \haz u^m)(t)\|^2 \\ \nonumber
  &=- \int_0^{t} (\D \phi(\ove u^n) - \D
  \phi(\ove u^m), \haz u^n - \haz u^m)\d r - \int_0^{t} (\ove r^n - \ove
    r^m, \haz u^n - \haz u^m)\d r \\ \nonumber
  &\leq C\left( \int_0^t \|\ove u^n - \ove u^m\| \, \|\haz u^n - \haz
  u^m\|\d r + \int_0^t (\|\ove r^n\|+\| \ove r^m\|) \|\haz u^n - \haz
  u^m\|\d r \right).
\end{align}

The difference of the piecewise linear and piecewise constant functions can be
estimated according to
\begin{align*}
  &\int_0^T \|\haz u^n - \ove u^n\|^2\d r 
	= \sum_{i=1}^{N^n}\int_{t^n_{i-1}}^{t^n_i} \|\haz u^n - \ove u^n\|^2\d r 
	= \sum_{i=1}^{N^n} \|u^n_i - u^n_{i-1}\|^2
  \left(\int_{t^n_{i-1}}^{t^n_i}(\ell^n_i)^2\d r\right) \nonumber\\
  &=\sum_{i=1}^{N^n} \frac{\tau^n_i}{3}
  \|u^n_i - u^n_{i-1}\|^2 \leq \frac{(\tau^n)^2}{3}
  \sum_{i=1}^{ N^n}\tau^n_i \left\|\frac{u^n_i - u^n_{i-1}}{\tau^n_i}
  \right\|^2 \leq C (\tau^n)^2 . 
\end{align*}
Using \eqref{eq:r} and the previous estimate, relation~\eqref{eq:un-um} gives
\begin{align*}
  \frac12& \|(\haz u^n - \haz u^m)(t)\|^2
  \le C\int_0^t\big(\|\ove u^n-\haz u^n\|^2 + \|\haz u^n-\haz u^m\|^2
	+ \|\haz u^m-\ove u^m\|^2\big)\d r \\
	&\phantom{xx}{}
	+ C\int_0^t \big(\|\ove r^n\|+\| \ove r^m\|\big) \|\haz u^n - \haz u^m\|\d r \\
	&\le C\int_0^t\|\haz u^n-\haz u^m\|^2\d r + C\big((\tau_i^n)^2 + (\tau_i^m)^2\big).
\end{align*}

We deduce from the Gronwall lemma that 
\begin{align}
\|\haz u^n -\haz u^m\|_{L^\infty(0,T;\Rz^d)} 
&\leq C(\tau^n + \tau^m) .\label{eq:progro2}
\end{align}

Then, again by arguing on the difference of
relations \eqref{eq:romerocompact} and using the local Lipschitz
continuity of $\D \phi$, we find that 
\begin{align*}
 \|(\haz u^n)' -(\haz u^m)'\|_{L^\infty(0,T;\Rz^d)}&\leq C \|\ove u^n - \ove
 u^m\|_{L^\infty(0,T;\Rz^d)} \\
&\phantom{xx}{}+ \|\ove
r^n\|_{L^\infty(0,T;\Rz^d)} + \|\ove r^m\|_{L^\infty(0,T;\Rz^d)}.
\end{align*}
Bounds \eqref{eq:r} and \eqref{eq:progro2} imply that
$$ \| \haz u^n - \haz u^m\|_{W^{1,\infty}(0,T;\Rz^d)} \leq C(\tau^n +
\tau^m).$$
By taking $m\to \infty$, we obtain the error estimate \eqref{eq:error_romero}.  We refer to \cite[Lemma~17.2.2, p.~695]{Attouch14} for the
analogous result for the Euler method. 

{\em Step 3: $\phi\in C^3(\Rz^d)$}.
We now turn to the proof of second-order consistency under condition
\eqref{eq:ansgar}. We conclude from condition \eqref{eq:ansgar} and the estimate
$\|u_i-u_{i-1}\|\le C\tau_i^n$ that the scheme \eqref{eq:romero02} can be written as
\begin{align*}
  \frac{u_i-u_{i-1}}{\tau_i^n} &+ \D\phi(u_i)
  = -\big(\phi(u_i) - \phi(u_{i-1}) - \D\phi(u_i)\cdot (u_i {-}
    u_{i-1})\big)\frac{u_i {-} u_{i-1}}{\|u_i {-} u_{i-1}\|^2}\\
  &= \left(\frac12   (u_i {-}
    u_{i-1})\cdot \D^2\phi(u_i)  (u_i {-}
    u_{i-1}) + {\rm O}(\|u_i{-}u_{i-1}\|^3)\right)\frac{u_i 
		- u_{i-1}}{\|u_i {-} u_{i-1}\|^2}\\
  &= \frac12 \D^2\phi(u_i)  (u_i {-}
    u_{i-1}) + {\rm O}((\tau^n_i)^2).
\end{align*}
In particular, $u_i$ solves
\begin{equation}
  \label{eq:ansgar2}
  u_i = u_{i-1} -\tau^n_i \D\phi(u_i) + \frac{ \tau^n_i}{2}\D^2\phi(u_i)  (u_i {-}
    u_{i-1}) + {\rm O}((\tau^n_i)^3).
\end{equation}
Let $w$ be the unique solution of $w'+ \D\phi(w)=0$ with
$w(t_{i-1})=u_{i-1}$. Then $w''= - \D^2\phi(w) w' =
\D^2\phi(w)\D\phi(w)\in C^1(\Rz^d)$. A Taylor expansion for some 
$\eta\in [t_{i-1},t_i]$ leads to
\begin{align}
  w(t_i) &= w(t_{i-1}) + \tau_i^n w'(t_{i}) -
  \frac{(\tau^n_i)^2}{2}w''(t_{i}) + \frac{(\tau^n_i)^3}{6}w'''(\eta) \nonumber \\
  &= u_{i-1} - \tau^n_i \D\phi( w(t_i))  -
  \frac{(\tau^n_i)^2}{2}\D^2\phi( w(t_i))\D\phi( w(t_i)) +
  \frac{(\tau^n_i)^3}{6}w'''(\eta)\nonumber\\ 
&= u_{i-1} - \tau^n_i \D\phi( w(t_i))  +
  \frac{\tau^n_i}{2}\D^2\phi( w(t_i))(w(t_i)- u_{i-1}) + {\rm
    O}((\tau^n_i)^3),
\label{eq:ansgar3}
\end{align}
where in the last step we used the fact that 
$w(t_i)- u_{i-1} = w(t_i)- w(t_{i-1}) = - \tau^n_i
\D \phi(w(t_i)) + {\rm O}((\tau^n_i)^2)$. 
Take now the difference of relations \eqref{eq:ansgar3} and
\eqref{eq:ansgar2}, multiply it by $ w(t_i) - u_i$ and use the local
Lipschitz continuity of the functions $\D\phi$ and $\xi \mapsto \D^2\phi(
\xi)(\xi{-} u_{i-1}) $ in order to obtain
$$\|w(t_i) - u_i\|^2 \leq \tau^n_i C \|w(t_i) - u_i\|^2 + C(\tau^n_i)^3 \|w(t_i) -
u_i\|.$$
Hence, for sufficiently small $\tau^n>0$, we infer that 
\begin{equation}
  \|w(t_i) - u_i\|\leq C (\tau^n_i)^3.\label{eq:ansgar4}  
\end{equation}

The error control \eqref{eq:error_romero2} follows from the stability
of Problem \eqref{eq:gf}. Indeed, let $u_1$, $u_2$ be two solutions of the
system $u'+\D\phi(u)=0$, it follows for all $0\leq s\leq t$ that
\begin{equation}
  \label{eq:stability}
  \|(u_1-u_2)(t)\| \leq \e^{L (t-s)} \|(u_1-u_2)(s)\|,
\end{equation}
where $L= \max_{t,i} \| \D^2\phi(u_i(t))\|$. We use  
\eqref{eq:ansgar4} and \eqref{eq:stability} to find that 
\begin{align}
 & \|u(t_i) - u_i\|  \nonumber\\
&\leq \|u(t_i) - w(t_i)\|+
  \|w(t_i) - u_{i}\|  \leq \e^{L\tau_i}\|u(t_{i-1})
  - u_{i-1}\| +
  \|w(t_i) - u_{i}\| \nonumber\\
&\leq \e^{L\tau_i}\|u(t_{i-1})
  - u_{i-1}\|+ C(\tau^n_i)^3 \leq \dots \leq
  Ci\max\{1,\e^{L T}\} (\tau^n_i)^3 \leq C(\tau^n)^2 \label{eq:stability2}
\end{align}
whence the bound $\|u - \haz u^n\|_{L^\infty}(0,T;\Rz^d)\leq C
(\tau^n)^2$ holds. Estimate \eqref{eq:error_romero2} follows from the
local Lipschitz continuity of $\D \phi$.
\end{proof}

 Some of the arguments of the proof of Theorem \ref{thm:r2} can be
adapted to the case of a gradient flow on an unbounded time
domain $t \in [0,\infty)$. In particular, by possibly asking $\phi$ to
be coercive, one can reproduce the convergence proof of point (i). On the
other hand, the error bounds at points (ii)-(iii) call for some
Gronwall argument, see \eqref{eq:progro2} and \eqref{eq:stability2},
which necessarily calls for the finiteness of the time interval   $[0,T]$. 

Note that condition \eqref{eq:ansgar} ensuring the second-order
convergence \eqref{eq:error_romero2} is always fulfilled in one
dimension ($d=1$) and is sharp, as the example  in Figure
\ref{fig:compare1}  $u'+u=0$, $u(0)=1$
shows.


 In several dimensions, condition \eqref{eq:ansgar} holds for radial
 functions $\phi$. On the other hand,  by testing on a
nonradial potential, one can check that the first-order convergence
rate in \eqref{eq:error_romero} is
sharp, as the choice 
\begin{equation}\label{eq:choice}
\phi(u)=u_1^2 + \tfrac14 u_2^2 \quad\text{for} \ u=(u_1,u_2)\in \Rz^2, \
u_0=(1,1),
\end{equation} with exact
solution $u(t)=(\exp(-2t),\exp(-t/2))$ shows, see Figure
  \ref{fig:errorgo}. 
Moreover, condition \eqref{eq:ansgar} holds for quadratic potentials.
	
\begin{figure}[ht]
\centering
\pgfdeclareimage[width=75mm]{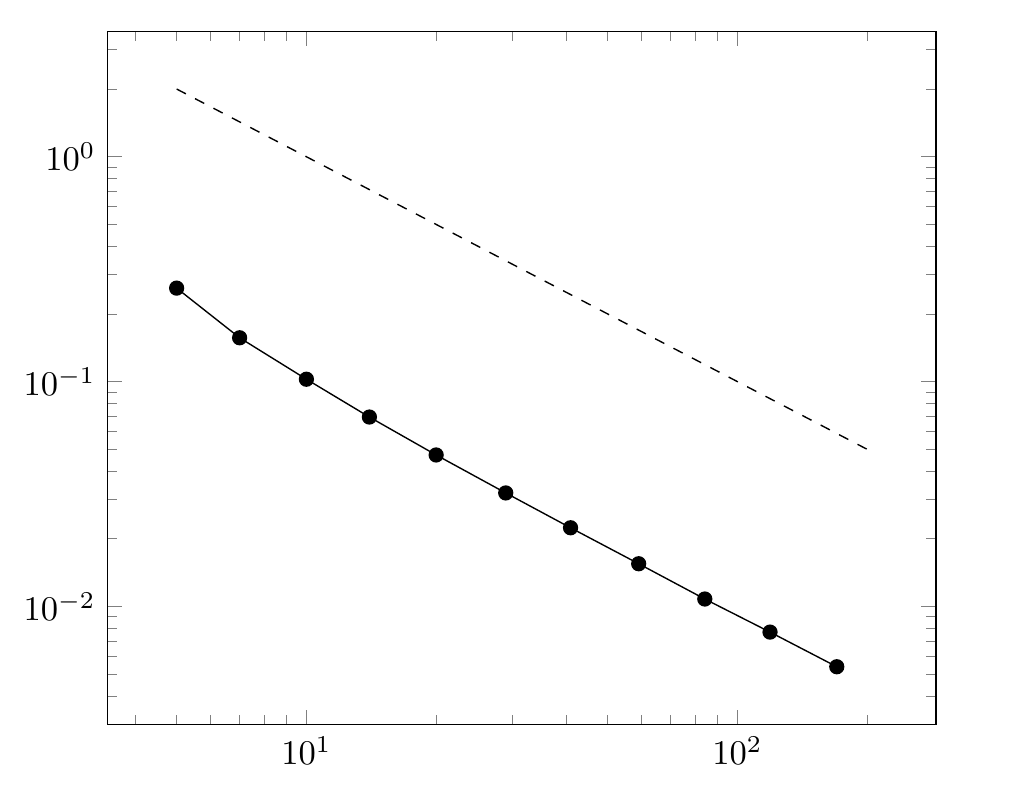}{errorgo} 
\pgfuseimage{errorgo}
\caption{$L^\infty$ error of the difference of the solution of the
Gonzalez scheme and the exact solution for the choice \eqref{eq:choice}
as a function of $1/\tau$ in log-log scale.
The dashed line corresponds to order $\tau$.
The order of convergence in \eqref{eq:error_romero} is sharp.}
\label{fig:errorgo}
\end{figure}

Note that in  Theorem \ref{thm:r2} one could approximate the initial data 
$u_{0}^n \to u_0$ as well. This
would still imply the convergence by including 
an additional term in the error estimate, taking into account the initial
error $\|u_0 - u_0^n\|$.

Before closing this section, let us comment on some shortcomings of
the Gonzalez scheme. First of all, the analysis of the Gonzalez scheme 
is at the moment restricted to
$C^{1}$ energies with compact sublevels, which in turn enforces a
finite-dimensional setting. 

In some nonsmooth cases, the existence of an
update could be conditional to the smallness of the time step. An
example in this direction is given by the scalar energy
$\phi(u)=-u+I_{[\sqrt{2},\infty)}(u)$, where $I_{[\sqrt{2},\infty)}$ stands
for the indicator function of the half-line $[\sqrt{2},\infty)$. For $u_0=2$,
the unique strong solution to \eqref{eq:gf} is
$u(t)=\max\{2-t,\sqrt{2}\}$. For all given {\it rational} time steps
$\tau>0$, one can however identify a number $i\in\Nz$ such that $2 - i\tau<\sqrt{2}<
2- (i-1)\tau$ and check that the Gonzalez scheme cannot be solved at
step $i$.

Eventually, the Gonzalez scheme seems not to be related to a variational
principle. The actual computation of the update from the previous step
involves the solution of the scalar energy equality as well as the
discussion on the alignment condition, which makes the incremental problem
a nonlinear system.


\section{The De Giorgi scheme}\label{sec:dg}

Let now the Hilbert space $H$ be general, possibly infinite-dimensional. 
We shall consider energies of the form $\phi = \phi_1 +
\phi_2$, where   $\phi_1:H \to
(-\infty,\infty] $ is convex, proper, lower semicontinuous and $\phi_2\in
C^{1,\alpha}_{\rm loc}(H)$ for some $\alpha \in (0,1]$. 
For the sake of notational simplicity, we assume 
$\partial \phi_1$ to be single-valued and remark that $u\in D(\partial
\phi) \equiv D(\partial \phi_1)$. The arguments easily extend to
not single-valued operators $\partial \phi_1$ with a bit notational intricacy.

In case $\phi$ is nonconvex, one can easily find
situations for which the De Giorgi scheme \eqref{eq:dg} has no solution. An example
in this direction is the scalar energy $\phi(u)=u(1-u)$ from
$u_{i-1}=0$. Indeed, in this case \eqref{eq:dg} reads as
$$
  u(1-u) +\frac{u^2}{2\tau}+\frac{\tau}{2}(1-2u)^2=0,
$$
which admits no real solution, independently of the choice of $\tau>0$.
We are hence forced to {\it generalize} the De Giorgi scheme by
allowing the possibility of solving  \eqref{eq:dg} with some 
tolerance. Define the functionals $G_i: D(\partial
\phi) \times  D(\partial \phi) \to (-\infty,\infty]$ as  
$$
  G_i(u,v) =  \phi(u) + \frac{\tau_i}{2}\left\| \frac{u  - v}{\tau_i}\right\|^2 +
  \frac{\tau_i}{2}\| \partial \phi(u )\|^2 -\phi(v).
$$ 
Given $u_{i-1}\in D(\partial\phi)$, one can find $u_i \in D(\partial
\phi)$ such that $(G_i(u_i,u_{i-1}))^+$ is arbitrarily small. 
In particular, we look for $u_i \in D(\partial \phi)$ such that
\begin{equation}
  \label{eq:dgxi}
  \phi(u_i) + \frac{\tau_i}{2}\left\| \frac{u_i - u_{i-1}}{\tau_i}\right\|^2 +
  \frac{\tau_i}{2}\|\partial \phi(u_i)\|^2 = \phi(u_{i-1}) + \rho_i
\end{equation}  
where the {\it residuals}
 $\rho_i$ are such that $\rho_i^+$ is small enough, see Theorem \ref{thm:convergence} below.

The following existence result holds.

\begin{theorem}[Existence for the De Giorgi scheme]\label{thm:existence}
Let $\phi = \phi_1 + \phi_2$  have compact sublevels, $\phi_1:H \to
(-\infty,\infty] $ be convex, proper, lower semicontinuous with
$\partial \phi_1$ being single-valued,  and $\phi_2\in C^{1,\alpha}_{\rm loc}(H)$ for 
$\alpha \in (0,1]$. Furthermore, let $u_{i-1}\in D(\partial \phi)$ be given with
$\partial \phi (u_{i-1}) \not = 0$. Then there
exists $u_i\in D(\partial \phi)$   with $u_i \not = u_{i-1}$ and
\begin{equation}\label{eq:dgtp}
  G_i(u_i, u_{i-1})\leq \frac{L}{1+\alpha}
  \|u_i - u_{i-1}\|^{1+\alpha},
\end{equation}
 where $L$ is the H\"older constant of $\D \phi_2$. In particular,
 \eqref{eq:dgxi} can be solved with $\rho_i \leq L \|u_i -
 u_{i-1}\|^{1+\alpha}/(1+\alpha)$. In case $\phi$ is convex, namely $\phi_2=0$, and
 $\|\partial \phi\|$ is strongly continuous along segments in $D(\partial \phi)$,
 one can find $u_i$ such that $G_i(u_i, u_{i-1})=0$.
\end{theorem}

In the convex case $\phi_2=0$, the functional $u\mapsto \|\partial \phi(u)\|$ is
lower semicontinuous along segments in $D(\partial
\phi)$. On the other hand, the continuity of $\|\partial \phi\|$ 
along segments assumed in the theorem may be available even in some nonsmooth 
situations. As an illustration of this
fact, let $H=L^2(\Omega)$ for some Lipschitz domain $\Omega\in \Rz^d$ 
and define $\phi$ to be the Dirichlet energy
$$ 
  \phi(u)= \frac12\int_\Omega |\nabla u|^2\d x
$$
for $u \in H^{1}(\Omega)$ and $\phi(u)=\infty$ otherwise, which is lower semicontinuous but
not continuous in $L^2$. Then
$D(\partial \phi) = H^2(\Omega)$ and $\partial \phi(u)= -\Delta
u$. Hence, for all $u_1$, $u_2 \in D(\partial \phi)$, the mapping
\begin{align*}
 & \lambda\in \Rz \mapsto \|\partial\phi(\lambda u_1 +(1-\lambda)
  u_2)\|= \| -\Delta(\lambda u_1 +(1-\lambda)
  u_2)\| \\
&= \left(\lambda^2 \| -\Delta u_1\|^2  + 2\lambda(1-\lambda)
  \int_\Omega \Delta u_1 \Delta u_2 \, \d x 
	+ (1-\lambda)^2 \|-\Delta u_2\|^2\right)^{1/2}  
\end{align*}
is continuous for $\lambda\in[0,1]$. 
Note that the continuity of
$\|\partial \phi\|$ along segments in $D(\partial \phi)$ is weaker than the 
strong continuity of $\partial \phi$.

\begin{proof}[Proof of Theorem \emph{\ref{thm:existence}}]
Let us start by considering the classical Euler scheme 
\begin{equation}
  \frac{u^e-u_{i-1}}{\tau_i} + \partial \phi(u^e)=0, \label{eq:euler}
\end{equation}
which can be solved by minimizing the function 
$$
  u \mapsto \frac{1}{2\tau_i}\| u - u_{i-1}\|^2 + \phi(u).
$$
Note that the latter function has compact sublevels and the direct method applies.
We readily check that $u_i:=u^e$ delivers  $u_i
\not = u_{i-1}$, for $u_{i-1} $ is not critical. In addition, one can
prove \eqref{eq:dgtp} as well.
To this aim, we test relation \eqref{eq:euler} with $u^e - u_{i-1}$ 
and use the convexity of $\phi_1$ and
the H\"older continuity of $\D\phi_2$ to find that
\begin{align}
  \tau_i &\left\| \frac{u^e-u_{i-1}}{\tau_i}\right\|^2 +\phi(u^e) -
  \phi(u_{i-1}) \nonumber\\
&\leq  \phi_2(u^e) - \phi_2(u_{i-1}) - (\D
  \phi_2(u^e),u^e - u_{i-1}) \nonumber\\
&= \int_0^1 \Big(\D\phi_2(\xi u^e {+} (1{-}\xi)u_{i-1}) - \D \phi_2
(u^e),u^e-u_{i-1}
\Big)\d \xi \nonumber\\
&\leq L\|u^e-u_{i-1}\|^{1+\alpha} \int_0^1(1-\xi)^\alpha\, \d \xi = \frac{L}{1+\alpha} 
\|u^e- u_{i-1}\|^{1+\alpha}.\label{eq:aa}
\end{align}
By using  again \eqref{eq:euler}, one finds that $ u^e$ fulfills 
\eqref{eq:dgtp} when choosing $u_i = u^e$.

In the convex case $\phi_2=0$, the choice $u_i=u^e$ does not
necessarily implies that
$G_i(u_i,u_{i-1})=0$.
In case $G_i(u^e,u_{i-1})< 0$
(check
$\phi(u)= u^2$ from $u_0=1$ for sufficiently small time steps), we can look for a point
$u_i$ along the segment  $u_\lambda
= \lambda u^e + (1-\lambda) u_{i-1}$ for $\lambda\in [0,1)$ such that $G_i(u_\lambda,u_{i-1})\equiv 0$.
Note that
the real function 
$$\lambda\mapsto g(\lambda):=G_i(u_\lambda,u_{i-1})$$
is well defined, for $D(\phi)$ and $D(\partial \phi)$ are
convex. Moreover, $\lambda\in [0,1] \mapsto \phi(u_\lambda)$ is convex and
lower semicontinuous, hence
continuous on $[0,1]$. The assumption on the continuity of $\|\partial
\phi\|$ along lines implies that $\lambda \mapsto
\|\partial\phi(u_\lambda)\|$ is continuous as well. Hence, $g$ is
continuous in $(0,1)$ and we find $\lambda \in (0,1)$ such that
$G_i(u_\lambda,u_{i-1})=g(\lambda)=0$ as $g(0)=\tau_i \| \partial \phi(u_{i-1})\|^2 /2>0$,
for $u_{i-1}$ is not singular, and $g(1)<0$. 
\end{proof}

As will be clear from the statement of Theorem \ref{thm:convergence}
below, the smallness of the residuals $\rho_i$ in \eqref{eq:dgxi}
will be instrumental to pass to the limit in the scheme. Theorem
\ref{thm:existence} claims that these can be as small as $L\|u_i
-u_{i-1}\|^{1+\alpha}$, which allows us to prove convergence. This
suggests to consider discrete solutions $u_i$ that {\it minimize}
the residuals. This corresponds to formulate a variational principle
of the form  
\begin{equation}
u_i \in \argmin_{u \in D(\partial \phi)} G_i(u,u_{i-1}).
\label{eq:min}
\end{equation}  
Recall the classical strong/weak closure property
$$ x_n \to x \ \text{strongly in $H$}, \ y_n \rightharpoonup y \ \text{weakly in
  $H$}, \ y_n \in \partial \phi_1(x_n) \ \Rightarrow \ y \in \partial \phi_1(x).$$

This minimum problem is solvable by the direct method, for  
 the sublevels of the function $u
\mapsto G_i(u,u_{i-1})$ are strongly compact. 
Indeed, one has $G_i(\cdot, u_{i-1})\geq \phi(\cdot) -\phi(u_{i-1})$
and the sublevels of $\phi$ are compact.

In particular, let $(u^n)$ be a strongly convergent minimizing sequence
such that $u^n\to u_i$ as $n\to\infty$. Then we deduce from 
the strong-weak closure of the subdifferential $\partial \phi_1$ that
$$ 
  \partial \phi(u^n) = \partial \phi_1(u^n) + \D \phi_2(u^n)
  \rightharpoonup  \partial \phi_1(u_i) + \D \phi_2(u_i) =\partial \phi(u_i) \quad
  \text{weakly in }H,
$$
and the lower semicontinuity of $u
\mapsto G_i(u,u_{i-1})$ follows. Clearly, $u_i \not = u_{i-1}$ if
$u_{i-1}$ is not critical. By inspecting the proof of
Theorem \ref{thm:existence}, we see that the minimization
in \eqref{eq:min} implies that $G_i(u,u_{i-1})\leq 0$, hence
fulfilling \eqref{eq:dgtp}. As such, Theorem \ref{thm:convergence}
below will ensure the convergence of the discrete solution obtained
via \eqref{eq:min} as well.

Let us now introduce our convergence result. To this aim, we specify
the notation with respect to a sequence of partitions 
$\{0=t_0^n<t_n^1<\dots<t_{N^n}^n=T\}$ as
$$
  G_i^n(u,v) =  \phi(u) + \frac{\tau_i^n}{2}\left\| \frac{u  - v}{\tau_i^n}\right\|^2 +
  \frac{\tau_i^n}{2}\| \partial \phi(u )\|^2 -\phi(v).
$$

 \begin{theorem}[Convergence for the De Giorgi scheme]\label{thm:convergence}
Under the assumptions of Theorem \emph{\ref{thm:existence}},
let $u_{i}^n\in D(\partial \phi)$ be such that $u_0^n=u_0$ and 
\begin{equation} \sum_{i=1}^{N^n}  
G_i^n(u_i^n,u_{i-1}^n)^+ \to 0 \ \ \text{as} \ \ n \to
\infty.\label{eq:condi}
\end{equation}
Then $\haz u^n \to u$ converges strongly in
$H^1(0,T;H)$, where $u$ solves the gradient-flow problem \eqref{eq:gf}.
 \end{theorem}

A comment on condition \eqref{eq:condi} is in order. For all $n\in\Nz$, 
let $u_{i}^n\in D(\partial
\phi)$ be the sequence fulfilling \eqref{eq:dgtp} from Theorem
\ref{thm:existence}. Then, 
\begin{align}
  \phi(u_{N^n}^n) &+\frac12\sum_{i=1}^{N^n}
      \tau_i^n\left\|\frac{u_i^n-u_{i-1}^n}{\tau_i^n}\right\|^2
     + \frac12\sum_{i=1}^{N^n}
      \tau_i^n\left\|\partial\phi(u^n_i)\right\|^2 - \phi(u_0)\nonumber \\
& =  \sum_{i=1}^{N^n} G_i^n(u_i^n,u_{i-1}^n) \leq
\frac{L(\tau^n)^\alpha}{1+\alpha}    
\sum_{i=1}^{N^n}\tau_i^n\left\|\frac{u_i^n-u_{i-1}^n}{\tau_i^n}\right\|^{1+\alpha}
\nonumber\\
& \leq C(\tau^n)^\alpha   \left(1 +
\sum_{i=1}^{N^n}\tau_i^n\left\|\frac{u_i^n-u_{i-1}^n}{\tau_i^n}\right\|^2 \right).
     \label{eq:bound0}
\end{align}
For sufficiently small values of $\tau^n>0$, we can absorb the 
last term on the right-hand side by the corresponding term on the left-hand side,
which leads to
$$
  \sum_{i=1}^{N^n}
  \tau_i^n\left\|\frac{u_i^n-u_{i-1}^n}{\tau_i^n}\right\|^2\leq C.
$$
Inserting this estimate into \eqref{eq:bound0} shows that
$$
  \sum_{i=1}^{N^n}G_i^n(u_i^n, u_{i-1}^n) \leq C (\tau^n)^\alpha,
$$ 
so that \eqref{eq:condi} holds. In particular, Theorem \ref{thm:convergence}
implies the convergence for the Euler scheme as well. More generally, condition
 \eqref{eq:condi}  can be seen as a criterion for checking
 convergence, independently of the procedure that produced the
 discrete solution. 

 \begin{proof}
Arguing as in \eqref{eq:bound0}, we have
   \begin{align}
  \phi(u_m^n) &+ \frac12\sum_{i=1}^m
      \tau_i^n\left\|\frac{u_i^n-u_{i-1}^n}{\tau_i^n}\right\|^2
     + \frac12\sum_{i=1}^m
      \tau_i^n\left\|\partial\phi(u^n_i)\right\|^2 - \phi(u_0)\nonumber \\
&\leq \sum_{i=1}^{N^n} G_i^n(u_i^n, u_{i-1}^n)^+\quad
\mbox{for all }m \leq N^n. 
     \label{eq:bound}
   \end{align}
As the right-hand side converges for $n \to \infty$, by assumption, one
deduces the bounds
$$ \sup_m \phi(u^n_m) + \frac12\sum_{i=1}^{N^n}
      \tau_i^n \left\|\frac{u_i^n-u_{i-1}^n}{\tau_i^n}\right\|^2
     +\frac12\sum_{i=1}^{N^n}
      \tau_i^n  \|\partial\phi(u^n_i)\|^2 \leq C.$$
Recall that the sublevels of  $\phi$ are assumed to be compact. We
can hence apply a diagonal
extraction argument (without relabeling) and obtain  
\begin{align}
  &\haz u^n \to u \quad \text{strongly in} \  \ C([0,T];H)
                   \ \ \text{and weakly in} \ \
                   H^1(0,T;H), \nonumber \\ 
&\ove u^n \to u \quad \text{strongly in} \  \ L^\infty(0,T;H), 
\nonumber \\ 
  &\partial\phi(\ove u^n) \to \partial\phi(u)\quad \text{weakly in} \  \ L^2(0,T;H). \label{conv3}
\end{align}
In order to identify the limit in \eqref{conv3}, 
we have used the strong-weak closure of the
subdifferential $\partial \phi$, which follows from the very definition of the subdifferential. 
In particular, 
$\ove u^n(t) \to u(t)$ for all $t \in (0,T)$.
Fix now  $t \in (0,T)$
and, for all $n$, choose $m$ such that $t^n_{m-1} <t \leq t^n_m$.
By passing to the limit inferior as $n\to \infty$ in estimate
\eqref{eq:bound} and using $\phi(u^n_m) =  \phi(\ove u^n(t))$, it follows that
\begin{align} 
\phi(u(t)) + \frac12 \int_0^t \|u'\|^2\d r 
+ \frac12 \int_0^t  \|\partial\phi(u)\|^2\d r \leq \phi (u_0). \nonumber 
\end{align}
Therefore, we can use the chain
rule \cite[Lemme 3.3, p. 73]{Brezis73} to conclude that
\begin{align*}
  \frac12 &\int_0^t \|u'\|^2\d r + \frac12  \int_0^t \|\partial\phi(u)\|^2 \d r
  \leq \phi (u_0) -
  \phi(u(t)) \\
  &= - \int_0^t (\phi\circ u)'\d r 
  =- \int_0^t (\partial\phi(u),u')\d r 
	\leq \frac12 \int_0^t \|u'\|^2\d r + \frac12 \int_0^t  \|\partial\phi(u)\|^2\d r.
\end{align*}
In particular, all these inequalities are actually equalities and 
$$\frac12   \|u'\|^2 + \frac12 \|\partial\phi(u)\|^2 = -(\partial\phi(u),u') \quad \text{a.e. in}
\ \ (0,T).$$
This implies that the equality
\eqref{eq:generic} holds and that 
   $u$ is a solution of
the gradient-flow problem \eqref{eq:gf}.
\end{proof}

 Theorem \ref{thm:convergence} can be formulated for gradient
flows on the whole semiline
$t \in [0,\infty)$ as well, by asking the convergence condition to be 
$$\sum_{i=1}^\infty G_i^n(u_i^n,u^n_{i-1})^+ \to 0 \ \ \text{as} \ \
n \to \infty.$$
Indeed, estimate \eqref{eq:bound0} holds in case $N^n=\infty$ as well. In
order to pass to the limit, one uses again the analogue of
\eqref{eq:bound}. A localization and diagonal-extraction argument
entails strong compactness. 

We present some error estimates for the De Giorgi scheme in finite
dimensions.

\begin{proposition}[Error control for the De Giorgi
  scheme in $\Rz^d$]\label{error} Let
  $\phi\in C^2(\Rz^d)$ be bounded from below and  $u_i$
  and $v_i$ fulfill $u_0=v_0$,
\begin{equation} G_i(u_i,u_{i-1})=0, \ \ \text{and} \ \ v_i\in\argmin
G_i(\cdot,v_{i-1}),\label{eq:schemes}
\end{equation}
  respectively. Then for all $i =1,\dots,N$
  \begin{align}
    &\| u(t_i) - v_i\|\leq C \tau, \label{eq:dge2}\\
 &\|  u(t_i)-u_i\|\leq C \tau^{1/2}, \label{eq:dge1}
  \end{align}
where $u$ is the unique solution of \eqref{eq:gf}.
\end{proposition}

\begin{proof} 
  Note first that \eqref{eq:gf} is uniquely
  solvable and that its solution $u$ is bounded. The
  existence of  $u_i$ is due to Theorem \ref{thm:existence} and that
  of $v_i$ follows by an application of the direct method.
  Owing to the
  energy equalities \eqref{eq:balance} and \eqref{eq:dg} and
  the minimality of $v_i$, 
	we conclude that $u$ is bounded uniformly in time 
	and that the values $u_i$ and $v_i$ are also bounded independently of $i$.  

{\em Step 1: proof of \eqref{eq:dge2}.}
Let us start by considering $v_i$. We show that the
consistency error is of order $\tau^2$. To this aim, let $w$
be the unique solution to $w'+\D\phi(w)=0$ with
$w(t_{i-1})=v_{i-1}$. Note that $w\in C^2$ and that $w''$ is bounded
for all times, depending on $\phi(v_{i-1})$ only (hence, just on $\phi(u_0)$). 
We deduce from the minimality of $v_i$ that 
$$
  0 = \D \phi(v_i) + \frac{1}{\tau_i}(v_i - v_{i-1}) + \tau_i
  \D^2\phi(v_i)\D \phi(v_i),
$$
which is equivalent to 
\begin{align}
  v_i &= v_{i-1} - \tau_i \D \phi(v_i) - \tau_i^2 \D ^2 \phi(v_i) \D
  \phi(v_i) .
\label{eq:const1}
\end{align}
On the other hand, a Taylor expansion ensures that
\begin{align}
w(t_i) = w(t_{i-1}) + \tau_i w'(t_{i})- \frac{\tau_i^2}{2}w''(\eta)
= v_{i-1} - \tau_i \D \phi (w(t_{i}))-
\frac{\tau_i^2}{2}w''(\eta) \label{eq:const2}
\end{align}
for some $\eta \in [t_{i-1},t_{i}]$. By taking the difference between
\eqref{eq:const2} and \eqref{eq:const1} and multiplying it by
$w(t_i)-v_i$, we find that 
\begin{align}
  \|&w(t_i)-v_i\|^2\nonumber\\
&\leq C\tau_i  \|w(t_i)-v_i\|^2 +  \left(\tau_i^2
  \|\D^2 \phi(v_i)\|\|\D \phi(v_i)\| + \frac{\tau_i^2}{2}\|w''(\eta)\| \right) \|w(t_i)-v_i\|\nonumber\\
&\leq C\tau_i  \|w(t_i)-v_i\|^2 +C \tau_i^2 \|w(t_i)-v_i\|, \nonumber
\end{align}
which, for sufficiently small values of $\tau_i>0$, implies that 
\begin{equation}
  \label{eq:const11}
  \|w(t_i)-v_i\| \leq C \tau_i^2.
\end{equation}
We now argue as in \eqref{eq:stability} and deduce from \eqref{eq:const11} that
\begin{align}
 \|u(t_i) - v_i\| &\leq \|u(t_i) - w(t_i)\|+
  \|w(t_i) - v_{i}\| \nonumber\\
&\leq \e^{L\tau_i}\|u(t_{i-1})
  - v_{i-1})\| +
  \|w(t_i) - v_{i}\| \nonumber\\
&\leq \e^{L\tau_i}\|u(t_{i-1})
  - v_{i-1})\|+ C\tau_i^2 \nonumber\\
&\leq \dots \leq
  Ci\max\{1,\e^{L T}\} \tau^2 \leq C\tau, \label{eq:const3}
\end{align}
from which \eqref{eq:dge2} follows.

{\em Step 2: proof of \eqref{eq:dge1}.}
Let us address the error control for $u_i$. We use $G_i(u_i,u_{i-1})=0$ to compute
\begin{align}
  \|u_i -u_{i-1} + \tau_i\D \phi(u_i)\|^2  
  &= \|u_i-u_{i-1}\|^2 + \tau_i^2\|\D\phi(u_i)\|^2 
	+ 2\tau_i \D\phi(u_i)\cdot (u_i - u_{i-1}) \nonumber\\
  &= 2\tau_i (-\phi(u_i) + \phi(u_{i-1}))
    + 2\tau_i \D\phi(u_i)\cdot (u_i - u_{i-1})\nonumber\\
  &= \tau_i \D^2\phi(\xi)(u_i - u_{i-1})\cdot (u_i-u_{i-1})
	\leq C\tau_i^3, \label{eq:const4}
\end{align}
where we used a Taylor expansion and $\xi$ belongs to the segment joining
$u_i$ and $u_{i-1}$. For the last inequality, we also used the estimate
$\|u_i-u_{i-1}\|\leq C\tau_i$.

Next, let $w$ be the unique solution to $w'+\D\phi(w)=0$ with
$w(t_{i-1})=u_{i-1}$. A Taylor expansion shows that
\begin{equation}
  \label{eq:const5}
  w(t_i) = u_{i-1} - \tau_i \D \phi(u_i) + \tau_i(\D\phi(u_i) -
  \D\phi(u_{i-1})) + \frac{\tau_i^2}{2}w''(\eta).
\end{equation}
By \eqref{eq:const5} and \eqref{eq:const4}, the estimate
\begin{align*}
  \|w(t_i) -u_i\|  
	&= \left\|u_{i-1} - \tau_i \D \phi(u_i) + \tau_i(\D\phi(u_i) -
  \D\phi(u_{i-1})) + \frac{\tau_i^2}{2}w''(\eta) -u_i\right\| \\
  &\leq \left\|u_i -u_{i-1} + \tau_i\D \phi(u_i)\right\| + C\tau_i^2 
	\leq C\tau_i^{3/2}
\end{align*}
follows. The error control \eqref{eq:dge1} can be proved by
 arguing similarly to   
\eqref{eq:const3}.
\end{proof}

 Differently from the convergence Theorem \ref{thm:convergence},
the proof of Theorem \ref{error} heavily relies on the finiteness of
the time interval $[0,T]$, for it hinges on a Gronwall-like argument,
see \eqref{eq:const3}. 

The convergence rates in \eqref{eq:dge2}-\eqref{eq:dge1} are sharp, as
 the one-dimensional test of Section \ref{sec:compare} shows,
see Figure \ref{fig:compare1}.

\section{Extensions}\label{sec:ext}

We discuss some extensions of the De Giorgi scheme  to other nonlinear
evolution equations.

\subsection{Generalized gradient flows}

The analysis of the De Giorgi scheme can be extended to 
generalized gradient flows, namely
\begin{equation}
  \label{eq:gf3}
  \partial \psi (u,u') + \partial \phi(u)\ni 0 \quad \text{for a.e.} \
  t\in (0,T), \quad u(0)=u_0.
\end{equation}
Here, $\psi: H \times H \to [0,\infty)$ and $\partial \psi
(u,u') $ denotes partial subdifferentiation with respect to the second
variable only and we recall that for simplicity, $\partial\phi$ is still
assumed to be single-valued.  
More precisely, we assume that
\begin{align}
&\bullet\ \forall u \in H:\ \psi(u,\cdot) \text{ is convex and lower
  semicontinuous}, \label{eq:psi1}\\
&\bullet\ \mbox{the mapping }H\times H\times H\to\Rz, 
  (u,v,w) \mapsto \psi(u,v) + \psi^*(u,w),  \nonumber \\
& \quad\ \text{is weakly lower semicontinuous}, \nonumber \\ 
&\bullet\ \exists c>0, \, p>1,\, \forall u,\, v,\, w\in H: \ \ \psi (u,v)+\psi^*(u,w)\geq c \|v\|^p+ c\|w\|^{p'},
 \label{eq:psi5}
\end{align}
where  $p'=p/(p-1)$ and the
Legendre-Fenchel conjugation is taken with respect to the second
variable only. An example for $\psi$ satisfying \eqref{eq:psi1}-\eqref{eq:psi5} is
$\psi(u,u') = \beta(u)|u'|^p$, where $p>1$ and $\beta$ is sufficiently smooth,
uniformly positive, and bounded. Note that, as a consequence of
\eqref{eq:psi5}, for any $u\in H$ one has
\begin{equation}
w \in \partial \psi(u,0)  \ \Leftrightarrow  \ 0=\psi(u,0)+\psi^*(u,w)
\stackrel{\eqref{eq:psi1}}{\geq} c|w|^q  \ \Rightarrow \ w =0.\label{eq:psi2}
\end{equation}

The analog of equality \eqref{eq:generic} for generalized gradient flows reads as
\begin{equation}
  \label{eq:generic3}
  \phi(u(t)) + \int_0^t \psi(u,u')\d r 
	+ \int_0^t \psi^*(u,-\partial \phi(u))\d r = \phi(u_0) 
\end{equation}
for all $t\in [0,T]$. 
This can be checked by equivalently rewriting
\eqref{eq:gf3} by using Fenchel's duality, $y\in \partial \psi(u,u')
\Leftrightarrow \psi(u,u') + \psi^*(u,y) = (y,u')$ with $y=-\partial \phi(u)$, as
$$ \psi(u,u')+ \psi^*(u,-\partial \phi(u)) + (\partial \phi(u),u')=0$$
and integrating in time.
Correspondingly, one modifies the functionals $G_i$ as follows
\begin{equation}
\widetilde G_i(u,v):=\phi(u) +  {\tau_i}\psi\left(v,\frac{u-v}{\tau_i}\right)
    +{\tau_i}\psi^* (v,-\partial  \phi(u) ) - \phi(v).\label{eq:gggf}
\end{equation}
Given the initial value $u_0 \in
    D(\phi)$, the De Giorgi scheme consists in finding $u_i \in
    D(\partial \phi)$  with $\widetilde G_i(u_i,u_{i-1})= \rho_i$ and
		sufficiently small $\rho_i^+$.
Note that the first occurrence in $\psi$ and $\psi^*$ in
\eqref{eq:gggf} is $v=u_{i-1}$, so that the scheme is
implicit.
Given a sequence of partitions, we use the notation
$$
 \widetilde G_i^n(u,v):=\phi(u) +  {\tau_i^n}\psi\left(v,\frac{u-v}{\tau_i^n}\right)
    +{\tau_i^n}\psi^* (v,-\partial  \phi(u) ) - \phi(v).
$$

\begin{theorem}[De Giorgi scheme for generalized gradient flows]
Assume \eqref{eq:psi1}-\eqref{eq:psi5} and let $\phi = \phi_1 + \phi_2$  have compact sublevels, $\phi_1:H \to
(-\infty,\infty] $ be convex, proper, lower semicontinuous with
$\partial \phi_1$ single-valued,  and $\phi_2\in C^{1,\alpha}_{\rm loc}(H)$ for $\alpha \in (0,1]$. Let $u_{i-1}\in D(\partial \phi)$ satisfy
$\partial \phi (u_{i-1}) \not = 0$. Then there
exists $u_i\in D(\partial \phi)$   with $u_i \not = u_{i-1}$ and
\begin{equation}\label{eq:dgtp2}
\widetilde G_i(u_i, u_{i-1})\leq \frac{L}{1+\alpha}
 \|u_i - u_{i-1}\|^{1+\alpha},
\end{equation}
 where $L$ is the H\"older constant of $\D \phi_2$. In case $\phi$ is convex, namely $\phi_2=0$, and
 $v\mapsto \psi^*(u, - \partial \phi(v))$ is continuous (with respect to the
strong topology)
along segments in $D(\partial \phi)$, for all $u\in H$, one can find
$u_i\not = u_{i-1}$ such that $\widetilde G_i(u_i, u_{i-1})=0$.

Let $u_{i}^n\in D(\partial
\phi)$ be such that $u_0^n=u_0$ and 
\begin{equation} \sum_{i=1}^{N^n}  
\widetilde G_i^n(u_i^n,u_{i-1}^n)^+ \to 0 \ \ \text{as} \ \ n \to
\infty.\label{eq:condi2}
\end{equation}
Then $\haz u^n \to u$ converges weakly in
$W^{1,p}(0,T;H)$, and $u$ solves the generalized gradient-flow problem \eqref{eq:gf3}.
\end{theorem}

\begin{proof} 
   By adapting the argument of Theorem \ref{thm:existence}, we wish to find
   $u^e$ solving the Euler scheme
 \begin{equation}
\partial \psi\left(u_{i-1},\frac{u^e-u_{i-1}}{\tau_i}\right)
+ \partial \phi(u^e)\ni 0.\label{eq:euler2}
\end{equation}
This can be done by letting
$$ u^e \in \argmin_{u\in H}  \left(\psi\left(u_{i-1},\frac{u-u_{i-1}}{\tau_i}\right)
+ \phi(u) \right).$$
This minimization problem is solvable as $\psi(u_{i-1},\cdot)$
and $\phi$ are lower semicontinuous, $\psi $ is nonnegative, 
and the sublevels of $\phi$ are
compact. Hence the Direct Method applies.
Note that the scheme in \eqref{eq:euler2} is implicit. 
As $\partial \phi(u_{i-1}) \not = 0$  and $\partial \psi(u_{i-1},0)=0$
from \eqref{eq:psi2}, we readily check that $u^e \not = u_{i-1}$.  

We rewrite relation \eqref{eq:euler2} equivalently as
$$\tau_i \psi\left(u_{i-1},\frac{u^e-u_{i-1}}{\tau_i}\right)  +
\tau_i \psi^*\left(u_{i-1},-\partial \phi(u^e) \right) +
\left(\frac{u^e-u_{i-1}}{\tau_i},\partial \phi(u^e) \right)=0.$$
By arguing as in \eqref{eq:aa}, we have that
$$ \Bigl(\frac{u^e-u_{i-1}}{\tau_i},\partial \phi(u^e)\Bigr)\geq
\phi(u^e)-\phi(u_{i-1})-\frac{L}{1+\alpha}\|u^e - u_{i-1}\|^{1+\alpha}.$$
In particular,
it turns out that \eqref{eq:dgtp2} holds for $u_i = u^e$.

In case $\phi_2=0$, we find that $\widetilde G_i(u^e,u_{i-1}) \leq
0$. Consider the map $\lambda \mapsto g(\lambda)
=\widetilde G_i(u_\lambda,u_{i-1})$ for $u_\lambda = \lambda u^e +
(1-\lambda)u_{i-1}$ and $\lambda\in[0,1)$. 
It is continuous, for  the functions $v \mapsto \phi(u_{i-1},v)$,
$v\mapsto \psi^*(u_{i-1}$, $-\partial \phi(v))$, and $v \mapsto\phi(v)$ are all 
continuous along the segments of $D(\partial \phi)$. Since $\partial
\phi(u_{i-1})\not =0$, we conclude that  
$$
  g(0)=\tau_i \psi (u_{i-1},0)  + \tau_i \psi^*(u_{i-1},
  -\partial \phi(u_{i-1})) \stackrel{\eqref{eq:psi5}}{\geq}
  c\tau_i\|\phi(u_{i-1})\|^q >0.
$$
In case $g(1)=0$ we have nothing to prove. If $g(1)<0$,
there exists $\lambda^* \in (0,1)$ such that $g(\lambda^*)= \widetilde
G_i(u_{\lambda^*},u_{i-1})=0$.
Now, let $u_i^n$ fulfill  \eqref{eq:condi2}. The  coercivity
\eqref{eq:psi5} implies that 
\begin{align}
  \phi(&u_m^n) + \sum_{i=1}^m \tau_i^n \left(c\left\| \frac{u_i^n -
  u_{i-1}^n}{\tau_i^n}\right\|^p + c\|\partial \phi(u_i^n)\|^{p'}
  -\frac1c\right) - \phi(u_0)\nonumber\\
&\leq
 \phi(u_m^n)  +\sum_{i=1}^m \tau_i^n \left(\psi\left(u_{i-1}^n,\frac{u_i^n -
  u_{i-1}^n}{\tau_i^n}\right) + \psi^*(u_{i-1}^n, - \partial \phi(u_i^n))
  \right) - \phi(u_0) \nonumber\\
&= \sum_{i=1}^{N^n}\widetilde
  G_i^n(u_i^n,u_{i-1}^n)^+\quad \mbox{for all }m \leq N^n. \label{eq:mandami}
\end{align}
As the right-hand side converges for $n\to \infty$, we infer that 
$\haz u^n$, $\ove u^n$, and
$\partial\phi(\ove u^n) $ are bounded in $W^{1,p}(0,T;H)$,
$L^\infty(0,T;H)$, and $L^{p'}(0,T;H)$, respectively, and that
$\phi(\ove u^n)$ is bounded. Hence, we can extract subsequences (not relabeled)
such that, as $n\to\infty$,
\begin{align*}
  \haz u^n \to u &\quad \text{strongly in} \  \ C([0,T];H)
                   \ \ \text{and weakly in} \ \
                   W^{1,p}(0,T;H),\\
\ove u^n, \, \underline u^n \to u &\quad \text{strongly in} \  \ L^\infty(0,T;H), \\
  \partial\phi(\ove u^n) \to \partial\phi(u)&\quad \text{weakly in} \  \ L^{p'}(0,T;H), 
\end{align*}
where $\underline u^n(t) = u_{i-1}$ for all $t \in
(t^n_{i-1},t^n_i]$. These convergences and  \eqref{eq:condi2} allow us to
pass to the limit inferior $n\to\infty$ in the last inequality in  
\eqref{eq:mandami}, giving
$$
  \phi(u(t)) + \int_0^t \psi(u,u')\d r + \int_0^t
  \psi^*(u,-\partial\phi(u))\d r - \phi(u_0)\leq 0
$$
for all $t \in [0,T]$. Eventually, the chain rule $(\phi \circ u)' = (\partial
\phi(u), u')$ (now in $W^{1,p}(0,T;$ $H)$) implies the energy
equality \eqref{eq:generic3}. We conclude that $u$ solves \eqref{eq:gf3}.
Indeed
\begin{align*}
 \int_0^t &\psi(u,u')\d r +\int_0^t\psi^*(u,-\partial \phi(u))\d r
\leq \phi(u_0)-\phi(u(t))=-\int_0^t (\phi \circ u)'\d r \\
 &=\int_0^t (-\partial\phi(u),u')\d r
 \leq \int_0^t \psi(u,u')\d r+\int_0^t \psi^*(u,-\partial\phi(u))\d r,
\end{align*}
so that all inequalities are actually equalities.
\end{proof}


\subsection{GENERIC flows} 

Another extension of the De Giorgi scheme concerns GE\-NER\-IC flows 
(General Equations for Non-Equilibrium  Reversible-Irreversible Coupling),
\begin{equation}
  \label{eq:gf4}
  u' = L \D E(u) -K \partial \phi(u)  \quad \text{for a.e.} \
  t\in (0,T), \quad u(0)=u_0.
\end{equation}
Here, $\phi$ plays the role of an entropy (up to the sign), being
nonincreasing in time; $E: H \to \Rz$ is an energy, being
conserved along trajectories; $K:H \to H$ is the so called Onsager operator, 
being linear, symmetric, and positive definite; and $L:H \to H$ is linear,
symplectic, and antiselfadjoint ($L^* = -L$).
One also assumes that the following compatibility conditions hold:
\begin{equation}
  \label{eq:gen}
  L^* \partial \phi (u) = K^* \D E(u)=0.
\end{equation}

The GENERIC formalism \cite{Grmela} 
is a systematic approach for the variational formulation of physical
models and is particularly tailored to the unified
treatment of coupled conservative and dissipative dynamics. 
As such, GENERIC has been applied to a variety of situations ranging from
complex fluids \cite{Grmela}, to dissipative quantum mechanics \cite{Mielke13}, to
thermomechanics \cite{Mielke11},
and to the Vlasov-Fokker-Planck equation \cite{Peletier}. 

Let $\psi(u)=(K^{-1}u,u)/2$.
The equivalent of \eqref{eq:generic} for GENERIC flows is 
\begin{equation}
  \label{eq:generic4}
  \phi(u(t)) + \int_0^t \psi(u'-L\D E(u))\d r + \int_0^t \psi^*(-\partial
  \phi(u))\d r = \phi(u_0) 
\end{equation}
for all $t\in [0,T]$. Indeed, given the compatibility \eqref{eq:gen},
equation \eqref{eq:generic4} and the equation in \eqref{eq:gf4} are
equivalent:
\begin{align*}
& \eqref{eq:gf4}  \ \Leftrightarrow\ 
  \psi(u'{-}L\D E(u))+\psi^*({-}\partial\phi(u)) 
	- \big( u'{-}L\D E(u), {-} \partial \phi(u)\big) =0 \ \text{a.e.}\\
& \Leftrightarrow  \ \frac{\d}{\d t} \phi\circ u + \psi(u'{-}L\D
  E(u))+\psi^*({-}\partial\phi(u)) =0 \ \text{a.e.} \  \Leftrightarrow \
  \eqref{eq:generic4}.
\end{align*}
 
The De Giorgi scheme can be extended to this case as well, by
considering the functional
$$ \ove G_i(u,v):=\phi(u) +  {\tau_i}\psi\left(\frac{u-v}{\tau_i} - L \D E(u)\right)
    +{\tau_i}\psi^* (-\partial  \phi(u) ) - \phi(v).$$

The existence of suitable De Giorgi solutions, namely $u_i \in
D(\partial \phi)$ such that $\ove G_i(u_i,u_{i-1})$ is small enough,
is still open.
By {\it assuming} that there exists a sequence $u_i$ fulfilling condition
\eqref{eq:condi} written for $\ove G$ instead of $G$, 
a necessary condition for convergence is that the real function
$(u,u')\mapsto \psi(u' - L\D E(u)) +\psi^*(-\partial\phi(u))$
is lower semicontinuous, which follows, for instance, if $\psi$ and
$\phi$ are convex and $E$ is $C^1$. 


\subsection{Curves of maximal slope in metric spaces}

The variational interpretation of the De Giorgi scheme from
\eqref{eq:min} can be extended to evolutions in metric spaces. Let $(X,d)$
be a complete metric space and $\phi:X \to [0,\infty]$ be lower semicontinuous. A locally absolutely continuous curve
$u:[0,T]\to X$ is said to be a {\it curve of maximal slope}  for the
functional $\phi$ if
$\phi\circ u$ is nonincreasing, $u(0)=u_0$, and 
\begin{equation}
  \label{eq:curves}
  \phi(u(t)) + \frac12\int_0^t |u'|^2(s)\d s + \frac12 \int_0^t |\partial
  \phi|^2(u(s))\, \d s = \phi(u_0)\quad \mbox{for all }t \in [0,T].
\end{equation}
It is beyond the purpose of this note to provide a comprehensive
discussion of this notion and the corresponding theory. We refer the
reader to the reference monograph by {\sc Ambrosio, Gigli, \& Savar\'e}
\cite{Ambrosio08} for details and limit ourselves in
observing that \eqref{eq:curves} exactly corresponds to \eqref{eq:generic} upon
replacing the norm of the time derivative with the {\it metric
  derivative}
\begin{equation}
|u'|(t):=\lim_{s\to t} \frac{d(u(s),u(t))}{|s-t|}
\label{eq:metricd}
\end{equation}
and the norm of the gradient of the functional by its {\it local
  slope} at $u \in D(\phi):=\{v \in X \ : \ \phi(v) <\infty\}$ as
$$ |\partial \phi|(u) = \limsup_{v\to
  u}\frac{(\phi(u)-\phi(v))^+}{d(u,v)}.$$
Let us recall from \cite[Theorem 1.1.2]{Ambrosio08} that the limit in \eqref{eq:metricd} exists
almost everywhere for trajectories in the class
\begin{align*}
  AC^2(0,T;X):=\bigg\{ & u:[0,T] \to X :\exists\, m \in L^2(0,T): 
	\forall\, 0\le s<t\le T: \\
	& d(u(s),u(t)) \leq \int_s^t m(r)\,\d r\bigg\}.
\end{align*}
A continuous curve $\gamma:[0,1]\to X$ is called a {\it constant-speed
  geodesic} \cite[Definition 2.4.2]{Ambrosio08}
if $$d(\gamma(t),\gamma(s)) = d(\gamma(0),\gamma(1))|t-s|\quad \mbox{for all }
s,\, t \in [0,1].$$
Given $\lambda\in \Rz$, we say that a functional $\eta: X \to (-\infty,\infty]$ 
is $\lambda$-{\it geodesically convex} 
\cite[Def. 2.4.3]{Ambrosio08} if for all points $u_0,\,
u_1\in D(\eta)=\{x \in X \ : \  \eta(x)<\infty\}$, 
there exists a constant-speed geodesic
$\gamma$ connecting $u_0$ and $u_1$ such that 
\begin{align*} 
& \eta(\gamma(t))  \leq (1-t) \eta(u_0) + t \eta(u_1)
  -\frac{\lambda}{2}t(1-t) d^2(u_0,u_1) \quad \mbox{for all }t \in [0,1].
\end{align*} 

The De Giorgi scheme in the metric setting is defined in terms of the
functionals $\haz G_i :D(|\partial \phi|) \times D(|\partial \phi|)
\to (-\infty,\infty]$, given by
$$\haz G_i(u,v) = \phi(u) +\frac{1}{2\tau_i}d^2(u,v) +
\frac{\tau_i}{2} |\partial \phi|^2(u) - \phi(v),$$
where $D(|\partial \phi|):=\{v \in D(\phi)\ : \ |\partial \phi|(v)
<\infty\}$. Given a sequence of partitions, we use as before the notation
$$
\haz G_i^n(u,v) = \phi(u) +\frac{1}{2\tau_i^n}d^2(u,v) +
\frac{\tau_i^n}{2} |\partial \phi|^2(u) - \phi(v).
$$
Our existence and convergence result reads as follows.

\begin{theorem}[De Giorgi scheme for curves of maximal
  slope]\label{thm:existence4} 
	Let $\tau_*>0$ be such that for all $v\in X$ and $\tau\le\tau_*$, the mapping
  $u\mapsto d^2(u,v)/(2\tau)+\phi(u)$ is
  $(1/\tau)$-geodesically convex. Assume that $\phi $ has compact sublevels, $u \mapsto
  |\partial\phi|(u)\in [0,\infty]$ is
  lower semicontinuous, and the chain-rule inequality \begin{equation}
  \label{eq:later}
  | (\phi\circ u)'| \leq |\partial \phi|(u)|u'| \quad \text{a.e. in}
  \ (0,T)
\end{equation}
holds for all $u \in AC^2(0,T;X)$ such that $|\partial \phi|(u) |u'| \in
L^1(0,T)$. Let $u_{i-1}\in
  D(|\partial \phi|)$ be given with
$|\partial \phi| (u_{i-1}) \not = 0$. Then, for all $\tau_i\leq\tau_*$, there
exists $u_i\in D(|\partial \phi|)$  with $u_i \not = u_{i-1}$ and
\begin{equation}\label{eq:dgtp4}
\haz G_i(u_i, u_{i-1})\leq 0.
\end{equation}

Let $u_{i}^n\in D(|\partial
\phi|)$ be such that $u_0^n=u_0$ and 
\begin{equation} \sum_{i=1}^{N^n}  
\haz G_i^n(u_i^n,u_{i-1}^n)^+ \to 0 \ \ \text{as} \ \ n \to\infty. \label{eq:condi4}
\end{equation}
Then $\ove u^n (t)\to u(t)$ for all $t \in [0,T]$ for a subsequence 
(which is not relabeled), where $u \in AC^2(0,T;X)$ is
a curve of maximal slope for the functional $\phi$.
\end{theorem}

A comment on these assumptions is in order. The geodesic convexity
of the squared distance corresponds to the assumption that the metric
space is {\it nonpositively curved}. This is the case of Euclidean and
Hilbert spaces, as well as Riemannian manifolds of
nonpositive sectional curvature \cite{Ambrosio08}.
As regards the functional $\phi$, let us recall Corollary 2.4.10 in \cite{Ambrosio08}, 
which ensures that $|\partial \phi|$ is
lower semicontinuous and the chain-rule inequality holds whenever
$\phi$ is $0$-geodesically convex. In particular, the
$(1/\tau)$-geodesical convexity of  $u\mapsto d^2(u,v)/(2\tau)+\phi(u)$
follows when $\phi$ is $0$-geodesically convex and the metric space
is nonpositively curved \cite[Remark 4.0.2]{Ambrosio08}.
 The assumptions can be
weakened by requiring compactness with respect to a weaker
topology, which could allow for the extension of the result to reflexive
Banach spaces \cite[Remark 2.0.5]{Ambrosio08}.

Theorem \ref{thm:existence4} shows that
\begin{align}
    u_i \in \argmin_{u \in D(|\partial \phi|)}\left(\phi(u) +
    \frac{1}{2\tau_i} d^2(u,u_{i-1}) +\frac{\tau_i}{2} |\partial
    \phi|^2(u) - \phi(u_{i-1})\right) \label{eq:mini}
\end{align}
for $i=1,\ldots,N$
is an alternative approximation for curves of maximal slope with respect 
to the classical Euler scheme
\begin{align} 
    u_i^e \in \argmin_{u \in D(|\partial \phi|)}\left(\phi(u) +
    \frac{1}{2\tau_i} d^2(u,u_{i-1}) - \phi(u_{i-1})\right) \label{eq:minie}
\end{align} 
for $i=1,\ldots,N$. 
 Note that the compactness of the sublevels of $\phi$ and the
lower semicontinuity of $|\partial\phi|$ imply that solutions to the
minimum problem \eqref{eq:mini} exist.

\begin{proof}[Proof of Theorem \ref{thm:existence4}]
Let $u_i^e$ be a solution to
\eqref{eq:minie}. We will use the {\it slope estimate} \cite[Lemma 3.1.3]{Ambrosio08}
\begin{equation}
  \label{eq:slope}
  |\partial\phi|(u^e_i) \leq \frac{1}{\tau_i}d(u_i^e,u_{i-1})
\end{equation}
as well as the following consequence of the $(1/\tau)$-geodesic
convexity of the functional
$u \mapsto d^2(u,u_{i-1})/(2\tau)+\phi(u)$ \cite[Theorem 4.1.2.ii]{Ambrosio08}
(choose $v=u=u_{i-1}$, $u_\tau=u_i^e$, and $\lambda=0$ in the notation of
the cited theorem):
\begin{equation}
  \label{eq:evi}
  \phi(u_i^e) +\frac{1}{\tau_i} d^2(u_i^e,u_{i-1}) - \phi(u_{i-1})\leq 0.
\end{equation}

Let $u_i$ solve \eqref{eq:mini}. We deduce from the minimality of $u_i$ and 
estimates \eqref{eq:slope} and \eqref{eq:evi} that 
\begin{align}
  \haz G_i(u_i,u_{i-1}) &\leq \haz G_i(u_i^e,u_{i-1}) \nonumber\\
&= \phi(u_i^e)
  +\frac{1}{2\tau_i} d^2(u_i^e,u_{i-1}) + \frac{\tau_i}{2}|\partial
  \phi|^2(u^e_i)- \phi(u_{i-1})\nonumber\\
& \leq \phi(u_i^e)
  +\frac{1}{2\tau_i} d^2(u_i^e,u_{i-1})   +\frac{1}{2\tau_i}
  d^2(u_i^e,u_{i-1}) - \phi(u_{i-1})\nonumber\\
&= \phi(u_i^e)
  +\frac{1}{\tau_i} d^2(u_i^e,u_{i-1})  - \phi(u_{i-1})
	\leq 0, \nonumber  
\end{align}
and \eqref{eq:dgtp4} follows. If $|\partial \phi|(u_{i-1})\not = 0$, we
have $\haz G_i(u_i,u_{i-1})\leq 0 < \tau_i|\partial \phi|^2(u_{i-1})/2 = \haz
G_i(u_{i-1},u_{i-1})$. Hence, $u_i\not = u_{i-1}$.


Next, let $u_i^n$ fulfill relation \eqref{eq:condi4}. We take the sum
for $i=1,\dots,m$:
\begin{align}
\phi(u_m^n) +  \sum_{i=1}^m
 \frac{d^2(u_i^n,u_{i-1}^n) }{2\tau_i^n}+  \sum_{i=1}^m \frac{\tau_i^n}{2}
|\partial \phi|^2(u_i^n) -\phi(u_0)
= \sum_{i=1}^m \haz
G_i^n(u^n_i,u^n_{i-1})^+.\label{eq:cms}
\end{align}
As the right-hand side is bounded uniformly with respect to $n$,
one can follow the proof of \cite[Theorem 2.3.3]{Ambrosio08} and extract
a subsequence $\ove u^n$ (not relabeled) such that $\ove u^n \to u$
pointwise in $[0,T]$, $|(\ove
u^n)'| \to |u'| $ weakly in $L^2(0,T)$, and for all $t \in [0,T]$,
$$\phi(u(t)) \leq \liminf_{n\to\infty}\phi(\ove u^n(t)), \quad
|\partial\phi|(u(t)) \leq \liminf_{n\to\infty}|\partial \phi|(\ove
u^n(t)).$$
Passing to the limit inferior in relation \eqref{eq:cms} leads to
\begin{align}
\phi(u(t)) &+ \frac12 \int_0^t |u'|^2(s)\, \d s + \frac12 \int_0^t |\partial
\phi|^2(u(s)) \, \d s - \phi(u_0)\nonumber\\
&\leq \liminf_{n\to \infty} \sum_{i=1}^{N^n}\haz
G_i^n(u_i^n,u_{i-1}^n)^+ = 0\label{eq:get}
\end{align}
for all $t \in [0,T]$.
 Eventually, we use \eqref{eq:get} and the  chain-rule inequality \eqref{eq:later}
 to obtain 
\begin{align*}
\frac12 \int_0^t &|u'|^2(s)\, \d s + \frac12 \int_0^t |\partial
\phi|^2(u(s)) \, \d s \leq -\phi(u(t)) +
  \phi(u_0) \\
&= -\int_0^t(\phi \circ u)'(s)\, \d s \leq \int_0^t|\partial
\phi|(u(s)) |u'|(s)\, \d s \\
&\leq \frac12 \int_0^t |u'|^2(s)\, \d s + \frac12 \int_0^t |\partial
\phi|^2(u(s)) \, \d s 
\end{align*}
so that all inequalities are actually equalities and $u$ solves \eqref{eq:curves}.
\end{proof}


\section{A comparison between the Gonzales and the De Giorgi
  scheme}\label{sec:compare}

Let us close this discussion by presenting a direct comparison of the output
of the Gonzales and the De Giorgi schemes in the case of the single
scalar ODE problem
$$u'+\lambda u = 0, \ \ u(0)=1$$
for $\lambda>0$ given. This
corresponds to the gradient flow of the uniformly convex
potential $\phi(u) = \lambda u^2/2$.

Given the time steps $\tau_i$, one readily finds the solutions $e_i$ of the
Euler scheme, $g_i$ of the Gonzales scheme, $u_i$ of the De Giorgi
scheme $G_i(u_i,u_{i-1})=0$, and $v_i$ of the De Giorgi scheme $\min
G_i(\cdot,v_{i-1})$, see \eqref{eq:schemes}, as 
\begin{align*}
  &e_i = \prod_{j=0}^i\frac{1}{1+\tau_i\lambda}, \quad g_i =
  \prod_{j=0}^i\frac{2-\tau_i \lambda}{2+\tau_i\lambda}, \\
& u_i
  =\prod_{j=0}^i\frac{1-(\tau_i\lambda)^{3/2}}{1+\tau_i\lambda+(\tau_i\lambda)^2}, \quad v_i =\prod_{j=0}^i\frac{1}{1+\tau_i\lambda+(\tau_i\lambda)^2}
 .
\end{align*}
Observe that $u_i$ and $g_i$ oscillate in sign for $\tau_i \lambda
>1$ and $\tau_i\lambda>2$, respectively, whereas $e_i$ and $v_i$ are always positive. 

Figure \ref{fig:compare1} records the
performance of the methods in terms of order of convergence for
$\lambda =1$.
\begin{figure}[ht]
 \centering 
\pgfdeclareimage[width=95mm]{compare1}{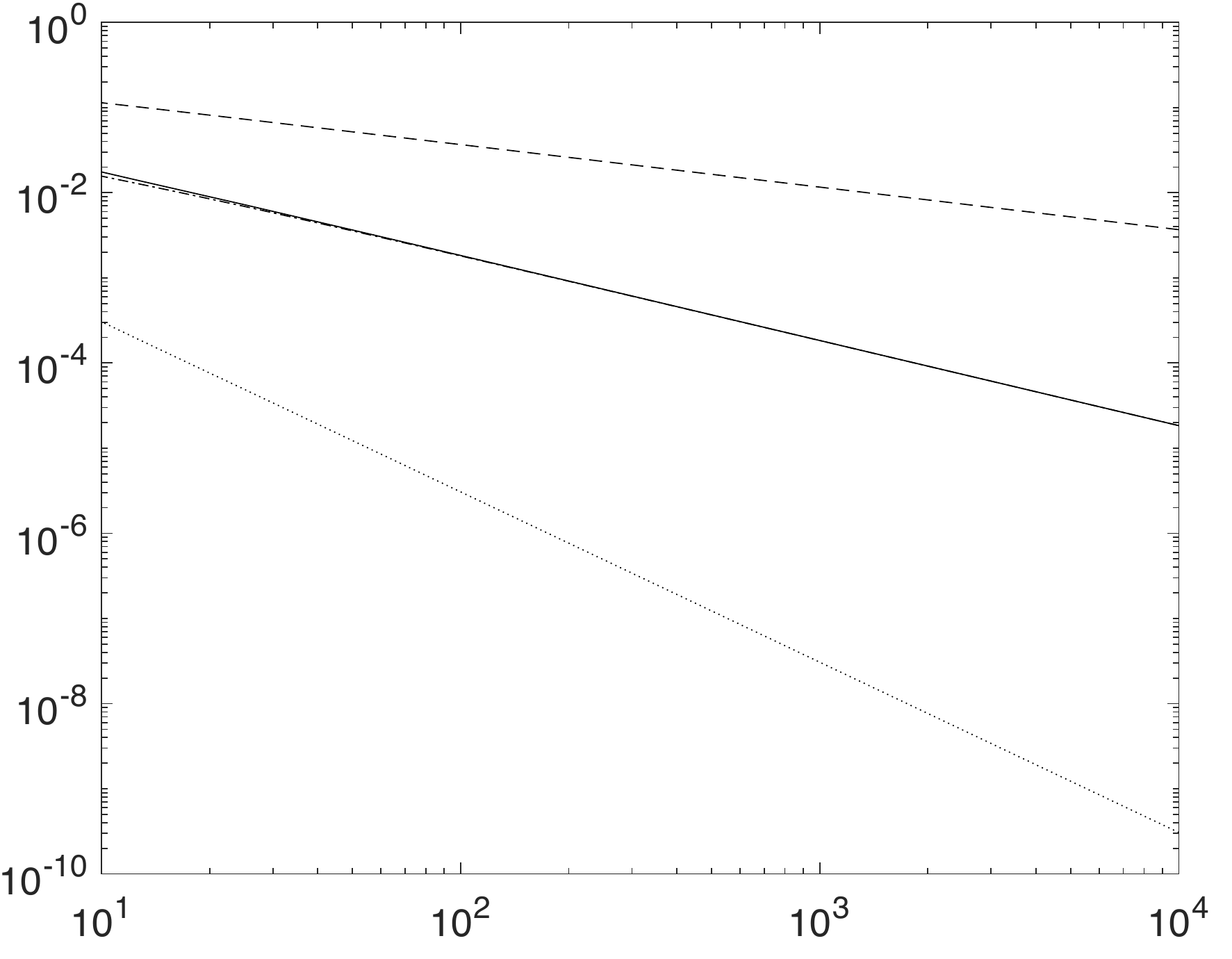} 
\pgfuseimage{compare1}
\caption{$L^\infty$ error on $[0,1]$ of the difference between the
  exact solution $u(t)=\e^{-t}$ of
$u'+ u=0$, $u(0)=1$ and the discrete solutions $e_i$ (solid line), $g_i$ (dotted line),
$v_i$ (dash-dotted line), and $u_i$ (dashed line) as a function of $1/\tau$ in 
log-log scale.}
\label{fig:compare1}
\end{figure}
In accordance to Theorem \ref{thm:r2}.iii, the Gonzales scheme is of order $\tau^2$, the Euler and the De~Giorgi scheme $v_i$ are of order $\tau$, and the De Giorgi scheme $u_i$ is
of order $\tau^{1/2}$. We recall that the
second-order convergence of the Gonzales scheme is limited to the
scalar case only. In more than one dimension, the Gonzales scheme is
only of
order $\tau$ if condition \eqref{eq:ansgar} does not hold, see Figure
\ref{fig:errorgo}.

Let us now assess the performance of the Gonzales and the De Giorgi schemes in order to compute the minimum of $\phi$.
We illustrate in Figure
\ref{fig:compare2} the reduction of the potential along iterations.
\begin{figure}[ht]
 \centering 
\pgfdeclareimage[width=65mm]{compare2a}{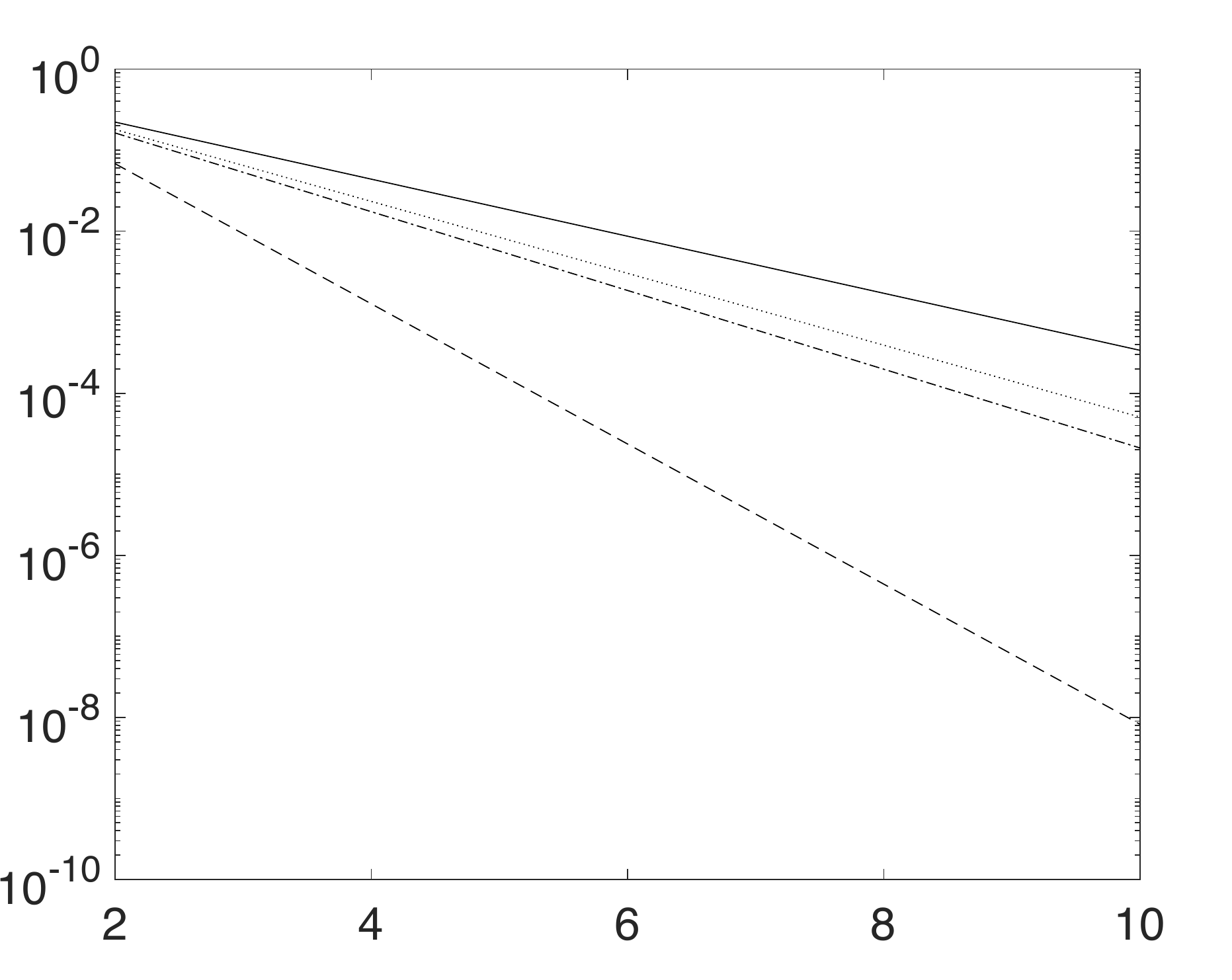} 
\pgfdeclareimage[width=65mm]{compare2b}{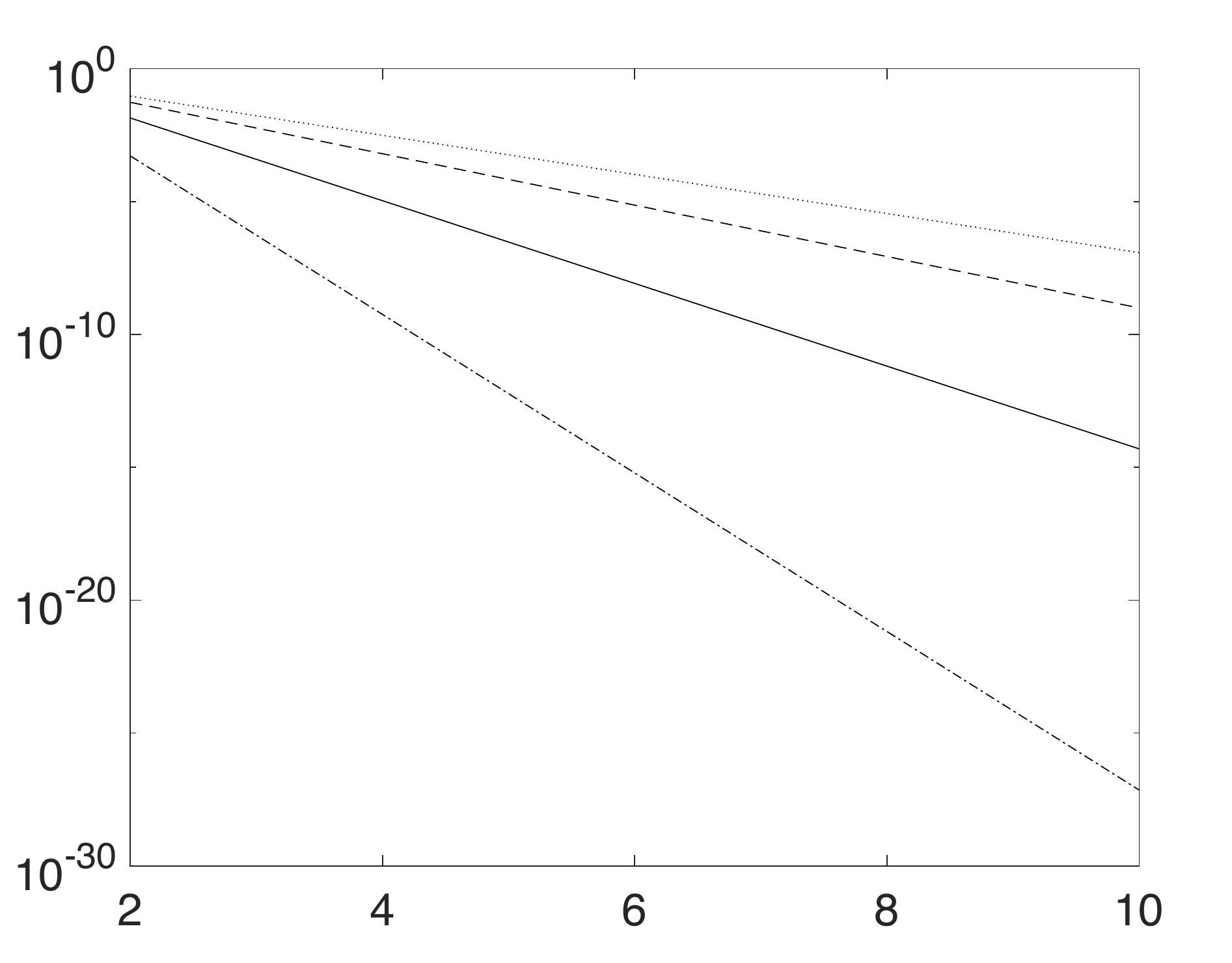} 
\centering
\pgfuseimage{compare2a}\pgfuseimage{compare2b}
\caption{Behavior of the potential $\phi(u)=\lambda u^2/2$ with $\lambda=1$ along iterations for
  $\tau\lambda=0.5$ (left) $\tau\lambda=5$ (right) and  in semilog scale; 
	$e_i$ (solid line), $g_i$ (dotted line),
  $v_i$ (dash-dotted line), and $u_i$ (dashed line).}
\label{fig:compare2}
\end{figure}
The performance of the Gonzales scheme and the De Giorgi scheme $u_i$
depends on the size of  $\tau \lambda$.
The De Giorgi scheme $v_i$ outperforms the Euler method regardless of the size of
$\tau \lambda$.  This fact is additionally illustrated in Figure
\ref{fig:compare3}, where we record the effect of the size of $\tau
\lambda$ on the effective reduction of the potential after a fixed
number of iterations.
\begin{figure}[ht]
 \centering 
\pgfdeclareimage[width=95mm]{compare3}{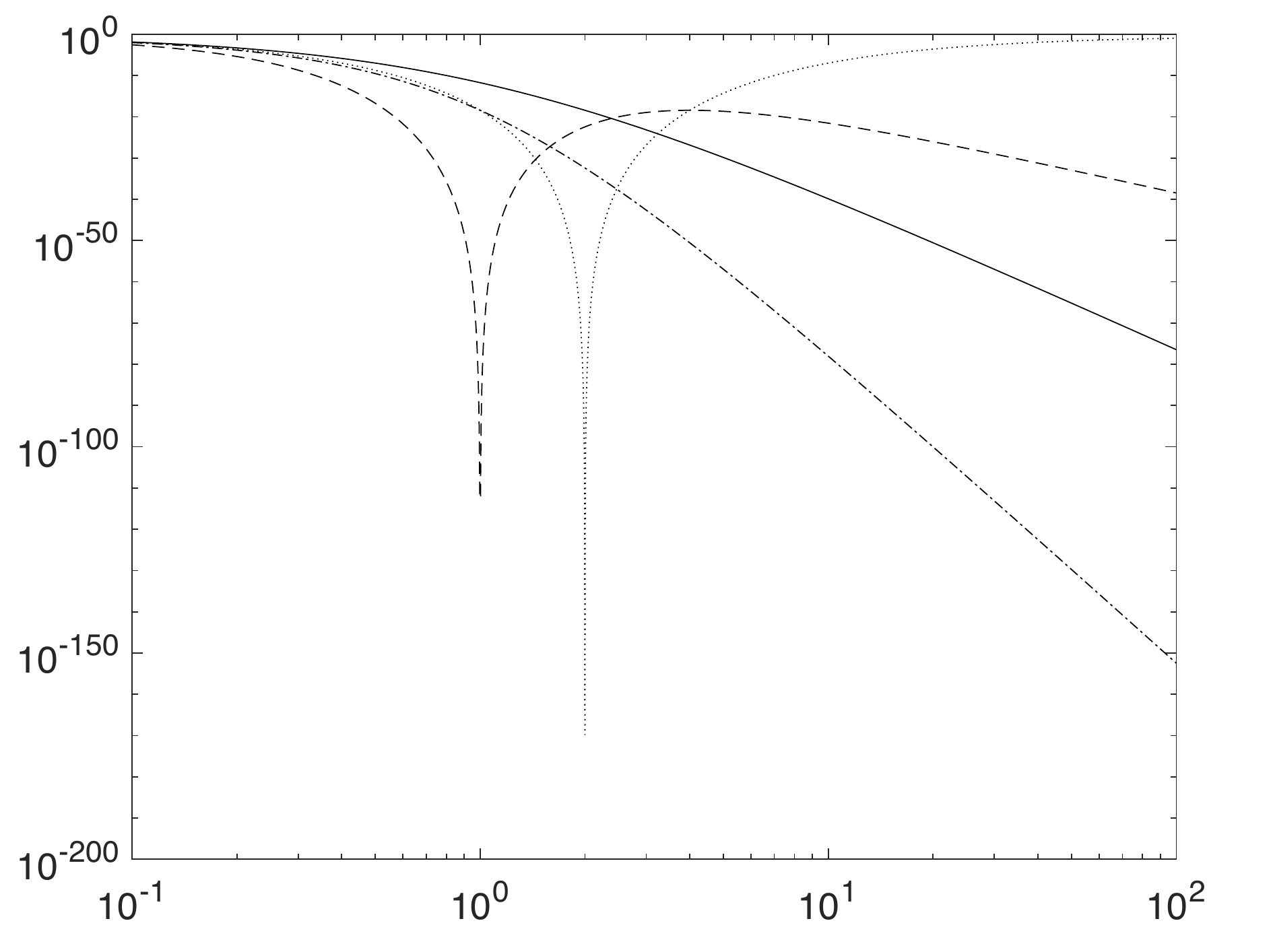} 
\centering
\pgfuseimage{compare3}
\caption{Reduction of the potential $\phi(u)=\lambda u^2/2$ with $\lambda=1$ after 20 iterations as a function
  of $ \tau\lambda$ in log-log scale; $e_i$ (solid line), $g_i$ (dotted line),
  $v_i$ (dash-dotted line), and $u_i$ (dashed line).}
\label{fig:compare3}
\end{figure}
One can observe how the Gonzales scheme delivers a strong energy reduction for $\tau \lambda<2$ and
virtually none for large $\tau \lambda$. Compared to the Euler method, both De
Giorgi schemes are more energy-reducing for large $\tau
\lambda$. While the De Giorgi scheme $u_i$ shows some singular
behavior at $\tau \lambda=1$, the De Giorgi scheme $v_i$ delivers an energy reduction regardless of the size of $\tau \lambda$, eventually outperforming the Euler method.


\section*{Acknowledgements}  
This research is supported by Austrian Science Fund (FWF) project F\,65. 
AJ is partially supported by the FWF projects P\,27352, P\,30000, and W\,1245. 
US is partially supported by the FWF projects I\,2375 and  P\,27052 and by the 
Vienna Science and Technology Fund (WWTF) through Project
MA14-009.  Some interesting remarks from the referees in the
direction of optimization algorithms are gratefully acknowledged. 



\begin{thebibliography}{99}

\bibitem{Ambrosio08}
L. Ambrosio, N. Gigli, G. Savar{\'e}. \emph{Gradient Flows in
  Metric Spaces and in the Space of Probability Measures}, second ed., 
	Birkh\"auser Verlag, Basel, 2008.


\bibitem{Attouch14}
H. Attouch, G. Buttazzo, G. Michaille. \emph{
Variational analysis in Sobolev and BV spaces. 
Applications to PDEs and optimization.} Second edition. MOS-SIAM
Series on Optimization, 17. Society for Industrial and Applied
Mathematics (SIAM), Philadelphia, PA; Mathematical Optimization
Society, Philadelphia, PA, 2014.

\bibitem{Attouch17}
H. Attouch, A. Cabot. Asymptotic stabilization of inertial gradient
dynamics with time dependent viscosity. {\it J. Differential Equ.} 263
(2017),  5412--5458.


\bibitem{Attouch18}
H. Attouch, Z. Chbani, J. Peypouquet, P. Redont. Fast convergence of
inertial dynamics and algorithms with asymptotic vanishing viscosity. {\it
Math. Program. Ser. B},  168 (2018), 123--175.


\bibitem{BKS08} A.~Bagirov, B.~Karas\"ozen, M.~Sezer. Discrete gradient method: 
derivative-free method for nonsmooth optimization, {\em J. Optim. Theory Appl.}
137 (2008), 317--334.



\bibitem{BC}
H. H. Bauschke, P. L. Combettes. {\it Convex analysis and monotone
operator theory in Hilbert spaces}. CMS Books in Mathematics/Ouvrages de Math\'ematiques de la SMC. Springer, New York, 2011.

\bibitem{Bert} 
D. P.  Bertsekas. {\it Convex optimization algorithms}. Athena
 Scientific, Belmont, MA, 2015.




 
\bibitem{Bot}
R. I. Bo\unichar{539}, E. R. Csetnek, S.C. L\'aszl\'o. Approaching nonsmooth
nonconvex minimization through second order proximal-gradient
dynamical systems. {\it J. Evol. Equ.} 18 (2018), 1291--1318.


\bibitem{Brezis71}
H.~Br\'ezis.
\newblock Monotonicity methods in {H}ilbert spaces and some application to
  nonlinear partial differential equations,
\newblock in: {\em Contributions to Nonlinear Functional Analysis}. Proc. Sympos. Univ.
  Wisconsin, Madison, pp.\ 101--156. Academic Press, New York, 1971.

\bibitem{Brezis73}
H.~Br\'ezis.
\newblock {\em Op\'erateurs maximaux monotones et semi-groupes de contractions
  dans les espaces de Hilbert}.
\newblock Number~5 in North Holland Math. Studies. North-Holland, Amsterdam,
  1973.


\bibitem{Cabot09}
A. Cabot, H. Engler, S. Gadat. On the long time behavior of second order differential
equations with asymptotically small dissipation. {\it
  Trans. Amer. Math. Soc.} 361 (2009), 5983--6017.


\bibitem{CGMMOOQ12} E.~Celledoni, V.~Grimm, R.~McLachlan, D.~McLaren, D.~O'Neale, 
B.~Owren, G.~Quispel. Preserving energy resp.\ dissipation in numerical PDEs 
using the ``Average Vector Field'' method, {\em J. Comput. Phys.} 231 (2012), 
6770--6789.


\bibitem{condette}
 {N.~Condette},  {C.~Melcher}, {E.~S\"uli}.
Spectral approximation of pattern-forming nonlinear evolution equations with
  double-well potentials of quadratic growth,
{\em Math. Comp.} {80} (2011), 205--223.

\bibitem{Crandall-Pazy69}
M.~G.~Crandall, A.~Pazy.
\newblock Semi-groups of nonlinear contractions and dissipative sets,
\newblock {\em J. Funct. Anal.} 3 (1969), 376--418.


\bibitem{DeGiorgi}
\newblock E.~{De Giorgi}, A.~Marino, M.~Tosques.
\newblock Problems of evolution in metric spaces and maximal decreasing curve,
\newblock \emph{Atti Accad. Naz. Lincei Rend. Cl. Sci. Fis. Mat. Natur. (8)},
  {68} (1980), 180--187.


\bibitem{Peletier}
M.~H.~Duong, M.~A.~Peletier, J.~Zimmer. GENERIC formalism of a
Vlasov--Fokker--Planck equation and connection to large-deviation
principles, {\it Nonlinearity}, 26 (2013), 2951--2971.

\bibitem{Eyr98} D.~Eyre. Unconditionally gradient stable time marching the 
Cahn--Hilliard equation, in: {\em Computational and Mathematical Models of 
Microstructural Evolution} (San Francisco, CA, 1998), pp.~39--46,
{\em Mater. Res. Soc. Sympos. Proc.} 529 (1998).

\bibitem{FuMa11} D.~Furihata, T.~Matsuo. {\em Discrete Variational Derivative
Method}, CRC Press, Boca Raton, USA, 2011.

\bibitem{HaLu14} E.~Hairer, C.~Lubich. Energy-diminishing integration of 
gradient systems, {\em IMA J. Numer. Anal.} 34 (2014), 452--461.

\bibitem{Garcia}
J.~C.~Garc\'\i a Orden, I.~Romero. Energy-Entropy-Momentum integration of
discrete thermo-visco-elastic dynamics, {\it Europ. J. Mech. Solids}, 32
(2012), 76--87.

\bibitem{Gonzales}
O.~Gonzalez.
Time integration and discrete Hamiltonian systems,
{\it J. Nonlin. Sci.} 6 (1996), 449--467.

\bibitem{Gonzales2}
O.~Gonzalez. Exact energy-momentum conserving algorithms for general
models in nonlinear elasticity, {\it Comp. Methods Appl. Mech. Engrg.} 190 (2000), 1763--1783. 

\bibitem{Grmela}
M.~Grmela, H.~C.~\"Ottinger. {Dynamics and thermodynamics of
  complex fluids. I. Development of a general formalism,} {\em
  Phys. Rev. E}, 56 (1997), 6620--6632.

\bibitem{Betsch}
M.~Gro\ss, M.~Bartelt, P.~Betsch. Structure-preserving time
integration of non-isothermal finite viscoelastic continua related to
variational formulations of continuum dynamics, {\it Comput. Mech.} (2017), 1--28.

\bibitem{HLW06} E.~Hairer, C.~Lubich, G.~Wanner. {\em Geometric Numerical 
Integration}, second edition, Springer, Berlin, 2006.



\bibitem{JuMi12} A.~J\"ungel, J.-P.~Mili\v{s}i\'{c}. Entropy dissipative 
one-leg multistep time approximations of nonlinear diffusive equations,
{\em Numer. Meth. Partial Diff. Eqs.} 31 (2015), 1119--1149.

\bibitem{JuSc17} A.~J\"ungel, S.~Schuchnigg. Entropy-dissipating semi-discrete 
Runge-Kutta schemes for nonlinear diffusion equations,
{\em Commun. Math. Sci.} 15 (2017), 27--53.

\bibitem{Komura67} Y.~K{\=o}mura.
\newblock Nonlinear semi-groups in {H}ilbert space,
\newblock {\em J. Math. Soc. Japan}, 19 (1967), 493--507.

\bibitem{LeTu17} G.~Legendre, G.~Turinici. Second-order in time schemes for 
gradient flows in Wasserstein and geodesic metric spaces. 
{\em C. R. Math. Acad. Sci. Paris}, 355 (2017), 345--353.

\bibitem{LWS17} K.~Liu, X.~Wu, W.~Shi. A linearly-fitted conservative (dissipative) 
scheme for efficiently solving conservative (dissipative) nonlinear wave PDES,
{\em J. Comput. Math.} 35 (2017), 780--800.


\bibitem{MaPl17} D.~Matthes, S.~Plazotta. A variational formulation of 
the BDF2 method for metric gradient flows. {\it ESAIM
  Math. Model. Numer. Anal.} 53 (2019),  145--172. 

\bibitem{McQu14} R.~McLachlan, G.~Quispel. Discrete gradient methods have an
energy conservation law, {\em Discrete Contin. Dyn. Sys.} 34 (2014), 1099--1104.

\bibitem{MQR99} R.~McLachlan, G.~Quispel, N.~Robidoux. Geometric integration
using discrete gradients, {\em Phil. Trans. R. Soc. Lond. A}, 357 (1999), 1021--1045.

\bibitem{Mielke11}
A.~Mielke. {Formulation of thermoelastic dissipative material behavior
using GENERIC}, {\it Contin. Mech. Thermodyn.} 23 (2011), 233--256.

 \bibitem{Mielke13}
 A.~Mielke. Dissipative quantum mechanics using GENERIC, Proc. of the
 conference on {\it Recent Trends in Dynamical Systems}, vol. 35 of
 Proceedings in Mathematics \& Statistics, Springer, 2013,
 pp. 555--585.


\bibitem{Nesterov}
 Yu. E. Nesterov. A method for solving the convex programming problem
 with convergence rate $O(1/k^2)$. {\it Dokl. Akad. Nauk SSSR},  269
 (1983),  543--547. 


\bibitem{Noch-Sav-Verdi2}
R.~Nochetto, G.~Savar\'e, C.~Verdi.
\newblock A posteriori error estimates for variable time-step discretization of
  nonlinear evolution equations,
\newblock {\em Commun. Pure Appl. Math.} 53 (2000), 525--589.


\bibitem{Pey}
J. Peypouquet. {\it Convex optimization in normed spaces. Theory, methods
and examples}. SpringerBriefs in
Optimization. Springer, Cham, 2015. 


\bibitem{QuMc08} G.~Quispel, D.~McLaren. A new class of energy-preserving 
numerical integration methods, {\em J.~Phys. A: Math. Theor.} 41 (2008), 045206,
7 pages.

\bibitem{Romero}
I.~Romero. 
Thermodynamically consistent time-stepping
algorithms for non-linear thermomechanical systems, {\it
  Internat. J. Numer. Methods Engrg.} 79 (2009), 706--732.

\bibitem{Romero3}
I.~Romero. 
Algorithms for coupled problems that preserve symmetries and the laws of thermodynamics. Part I: monolithic integrators and their application to finite 
strain thermoelasticity.  
{\it Comput. Methods Appl. Mech. Engrg.} 199 (2010), 1841--1858.
	
\bibitem{Romero2}
I.~Romero. 
Algorithms for coupled problems that preserve symmetries and the laws
of thermodynamics. Part II: fractional step methods. {\it
  Comput. Methods Appl. Mech. Engrg.} 199 (2010), 2235--2248.  

\bibitem{Rossi-Savare06}
R.~Rossi, G.~Savar\'e.
\newblock {Gradient flows of non convex
functionals in Hilbert spaces and applications}.
\newblock {\em ESAIM Control Optim. Calc. Var.} 12 (2006), 564--614.

\bibitem{SMSF15} S.~Sato, T.~Matsuo, H.~Suzuki, D.~Furihata. A Lyapunov-type
theorem for dissipative numerical integrators with adaptive time-stepping,
{\em SIAM J. Numer. Anal.} 53 (2015), 2505--2518.


\bibitem{Shi}
B. Shi, S. S. Du, M. I. Jordan, W. J. Su.  Understanding
the acceleration phenomenon via high-resolution differential equations.
Submitted for publication, 2018. arXiv/2440124/2018.


\end{thebibliography}
\end{document}